\documentclass[final,5p,times,8pt]{elsarticle}
\usepackage[centertags]{amsmath}
\usepackage{amsmath,amsfonts,amssymb,amsthm,epstopdf,newlfont,graphicx}
\usepackage{fixltx2e,booktabs,bm,enumitem,setspace,empheq,cancel,tikz}
\usetikzlibrary{shapes.geometric, arrows}
\usepackage{tabularx}
\usepackage{animate}
\usepackage{longtable}
\usepackage[raggedrightboxes]{ragged2e}
\usepackage{algorithm}
\usepackage{float}
\usepackage[short,nocomma]{optidef}
\usepackage{algpseudocode}
\usepackage{subcaption}
\usepackage[labelfont=bf,format=plain,justification=raggedright,singlelinecheck=false]{caption}
\usepackage{threeparttable,multirow}
\usepackage{amsmath,accents}
\usepackage[LGR,T1]{fontenc}

\usepackage[american]{babel}
\usepackage{natbib}
\biboptions{comma,square,sort&compress}
\usepackage{appendix}
\usepackage[colorlinks=true]{hyperref}
\def\diag{\mathop{\mathrm{diag}}}

\newcommand\hcancel[2][black]{\setbox0=\hbox{$#2$}%
\rlap{\raisebox{.25\ht0}{\textcolor{#1}{\rule{0.7\wd0}{0.75pt}}}}#2} 
\newcommand\hcancelt[2][black]{\setbox0=\hbox{$#2$}%
\rlap{\raisebox{.25\ht0}{\textcolor{#1}{\hspace{0.3mm}\rule{0.7\wd0}{0.75pt}}}}#2} 
\newtheorem{thm}{Theorem}[section]
\newtheorem{lem}{Lemma}[section]

\newtheorem{rem}{Remark}[section]
\theoremstyle{definition}

\numberwithin{algorithm}{section}
\numberwithin{equation}{section}
\renewcommand{\theequation}{\thesection.\arabic{equation}}
\def\simgt{\,\hbox{\lower0.6ex\hbox{$>$}\llap{\raise0.3ex\hbox{$\sim$}}}\,}
\def\simlt{\,\hbox{\lower0.6ex\hbox{$<$}\llap{\raise0.3ex\hbox{$\sim$}}}\,}
\def\simgteq{\,\hbox{\lower0.6ex\hbox{$\ge$}\llap{\raise0.6ex\hbox{$\sim$}}}\,}
\def\simlteq{\,\hbox{\lower0.6ex\hbox{$\le$}\llap{\raise0.6ex\hbox{$\sim$}}}\,}
\def\applteq{\,\hbox{\lower0.6ex\hbox{$\le$}\llap{\raise0.8ex\hbox{$\approx$}}}\,}
\def\applt{\,\hbox{\lower0.6ex\hbox{$<$}\llap{\raise0.5ex\hbox{$\approx$}}}\,}
\DeclareMathAlphabet\mathbfcal{OMS}{cmsy}{b}{n}
\DeclareMathAlphabet{\mathpzc}{OT1}{pzc}{m}{it}
\DeclareMathAlphabet\euscr{U}{eus}{m}{n}

\DeclareMathOperator\supp{supp}
\makeatletter
\def\user@resume{resume}
\def\user@intermezzo{intermezzo}
\def\simgt{\,\hbox{\lower0.6ex\hbox{$>$}\llap{\raise0.4ex\hbox{$\sim$}}}\,}
\def\simlt{\,\hbox{\lower0.6ex\hbox{$<$}\llap{\raise0.4ex\hbox{$\sim$}}}\,}
\def\simgteq{\,\hbox{\lower0.6ex\hbox{$\ge$}\llap{\raise0.6ex\hbox{$\sim$}}}\,}
\def\simlteq{\,\hbox{\lower0.6ex\hbox{$\le$}\llap{\raise0.6ex\hbox{$\sim$}}}\,}
\newcounter{previousequation}
\newcounter{lastsubequation}
\newcounter{savedparentequation}
\renewenvironment{subequations}[1][]{%
      \def\user@decides{#1}%
      \setcounter{previousequation}{\value{equation}}%
      \ifx\user@decides\user@resume 
           \setcounter{equation}{\value{savedparentequation}}%
      \else  
      \ifx\user@decides\user@intermezzo
           \refstepcounter{equation}%
      \else
           \setcounter{lastsubequation}{0}%
           \refstepcounter{equation}%
      \fi\fi
      \protected@edef\theHparentequation{%
          \@ifundefined {theHequation}\theequation \theHequation}%
      \protected@edef\theparentequation{\theequation}%
      \setcounter{parentequation}{\value{equation}}%
      \ifx\user@decides\user@resume 
           \setcounter{equation}{\value{lastsubequation}}%
         \else
           \setcounter{equation}{0}%
      \fi
      \def\theequation  {\theparentequation  \alph{equation}}%
      \def\theHequation {\theHparentequation \alph{equation}}%
      \ignorespaces
}{%
  \ifx\user@decides\user@resume
       \setcounter{lastsubequation}{\value{equation}}%
       \setcounter{equation}{\value{previousequation}}%
  \else
  \ifx\user@decides\user@intermezzo
       \setcounter{equation}{\value{parentequation}}%
  \else
       \setcounter{lastsubequation}{\value{equation}}%
       \setcounter{savedparentequation}{\value{parentequation}}%
       \setcounter{equation}{\value{parentequation}}%
  \fi\fi
  \ignorespacesafterend
}
\makeatother
\newcommand{\C}[1]{\mathcal{#1}}

\newcommand{\F}[1]{\mathbf{#1}}

\newcommand{\FR}[1]{\mathfrak{#1}}
\newcommand{\MB}[1]{\mathbb{#1}}

\newcommand{\ME}[1]{\euscr{#1}}

\newcommand{\MBG}{\MB{G}}
\newcommand{\MBSG}{\hat{\MBG}}
\newcommand{\MBR}{\mathbb{R}}
\newcommand{\MBC}{\mathbb{C}}

\newcommand{\MBRP}{\MBR^+}
\newcommand{\MBRzer}{\MBR_0}
\newcommand{\MBRczer}{\MBR_{\hcancel{0}}}
\newcommand{\MBRzerP}{\MBRzer^+}
\newcommand{\MBRmhzer}{\MBR_{-1/2}^-}
\newcommand{\MBZ}{\mathbb{Z}}
\newcommand{\MBZP}{\MBZ^+}
\newcommand{\MBZzer}{\MBZ_0}
\newcommand{\MBZzerP}{\MBZzer^+}
\newcommand{\MBZe}{\MBZ_e}
\newcommand{\MBZeP}{\MBZe^+}

\newcommand{\MFI}{\mathfrak{I}}

\newcommand{\MBJ}{\mathbb{J}}
\newcommand{\MBJP}{\mathbb{J}^+}

\newcommand{\MBN}{\mathbb{N}}
\newcommand{\MFF}{\mathfrak{F}}

\newcommand{\MFC}{\mathfrak{C}}

\newcommand{\cancbra}[1]{\hcancel{[}#1\hcancelt{]}}

\newcommand{\hG}{\hat{G}}
\newcommand{\hK}{\hat{K}}

\newcommand{\hvartheta}{\hat{\vartheta}}
\newcommand{\bmx}{\bm{x}}

\newcommand{\bmt}{\bm{t}}

\newcommand{\bmz}{\bm{z}}

\newcommand{\bmv}{\bm{v}}

\newcommand{\bmy}{\bm{y}}
\newcommand{\hFP}{\hat{\F{P}}}
\newcommand{\hFD}{\hat{\F{D}}}

\newcommand{\bmh}{\bm{h}}

\newcommand{\bmzer}{\bm{\mathit{0}}}
\newcommand{\bmone}{\bm{\mathit{1}}}

\newcommand{\hbmx}{\hat\bmx}

\newcommand{\hx}{\hat{x}}

\newcommand{\hchi}{\hat{\chi}}

\newcommand{\hlambdabar}{\hat \lambdabar}
\newcommand{\hvarpi}{\hat \varpi}

\newcommand{\FOmega}{\F{\Omega}}
\newcommand{\IFOmega}{\FOmega^{\circ}}

\newcommand{\foralla}{\,\forall_{\mkern-6mu a}\,}
\newcommand{\forallaa}{\,\forall_{\mkern-6mu aa}\,}
\newcommand{\foralle}{\,\forall_{\mkern-6mu e}\,}
\newcommand{\foralls}{\,\forall_{\mkern-6mu s}\,}
\newcommand{\forallS}{\,\forall_{\mkern-4mu rs}\,}
\newcommand{\forallL}{\,\forall_{\mkern-4mu rl}\,}
\newcommand{\Def}[1]{\text{Def}\left(#1\right)}

\newcommand{\trp}[1]{\text{trp}\left(#1\right)}

\newcommand{\resh}[3]{\text{resh}_{#1,#2}\left(#3\right)}
\newcommand{\reshCB}[3]{\text{resh}_{#1,#2}\left[#3\right]}

\newcommand{\reshTCB}[3]{\text{resh}_{#1,#2}^{\top}\left[#3\right]}
\newcommand{\reshs}[2]{\text{resh}_{#1}\left(#2\right)}

\newcommand{\Diag}[1]{\diag\left(#1\right)}

\newcommand{\sigmabar}{\mathord{\sigma\kern-0.6em\raisebox{-1ex}{$\bar{\phantom{\sigma}}$}}}
\newcommand{\sigmabarmax}{\mathord{\sigma_{\max}\kern-1.9em\raisebox{-0.8ex}{$\bar{\phantom{\sigma}}$}}\;\;\;\;\;}
\newcommand{\sigmabarmin}{\mathord{\sigma_{\min}\kern-1.75em\raisebox{-0.8ex}{$\bar{\phantom{\sigma}}$}}\;\;\;\;\;}
\newcommand{\CapD}[3]{\,{}^{c}D_{#1}^{#2}#3}
\newcommand{\CapDE}[3]{\,{}^{E}D_{#1}^{#2}#3}
\newcommand{\CapDEN}[3]{\,{}_{n_q,\lambda_q}^{n,\lambda,E}D_{#1}^{#2}#3}

\newcommand{\CapIM}[2]{\,{}^{E}\F{Q}_{#1}^{#2}}
\newcommand{\hCapIM}[2]{\,{}^{E}\hat{\F{Q}}_{#1}^{#2}}
\newcommand{\EE}[2]{{#1}^{\kern-.15em\circ #2}}
\newcommand{\Part}[1]{\partial_{#1}}
\newcommand{\NPart}[2]{\partial_{#1}^{\,#2}}
\usepackage{mathtools}
\mathtoolsset{showonlyrefs}
\makeatletter
\def\BState{\State\hskip-\ALG@thistlm}
\makeatother
\MakeRobust{\eqref}
\errorcontextlines\maxdimen

\makeatletter
    \newcommand*{\algrule}[1][\algorithmicindent]{\makebox[#1][l]{\hspace*{.5em}\thealgruleextra\vrule height \thealgruleheight depth \thealgruledepth}}%
\newcommand*{\thealgruleextra}{}
\newcommand*{\thealgruleheight}{.75\baselineskip}
\newcommand*{\thealgruledepth}{.25\baselineskip}

\newcount\ALG@printindent@tempcnta
\def\ALG@printindent{%
    \ifnum \theALG@nested>0
        \ifx\ALG@text\ALG@x@notext
        \else
            \unskip
            \addvspace{-1pt}
            \ALG@printindent@tempcnta=1
            \loop
                \algrule[\csname ALG@ind@\the\ALG@printindent@tempcnta\endcsname]%
                \advance \ALG@printindent@tempcnta 1
            \ifnum \ALG@printindent@tempcnta<\numexpr\theALG@nested+1\relax
            \repeat
        \fi
    \fi
    }%
\usepackage{etoolbox}
\patchcmd{\ALG@doentity}{\noindent\hskip\ALG@tlm}{\ALG@printindent}{}{\errmessage{failed to patch}}
\makeatother

\newbox\statebox
\newcommand{\myState}[1]{%
    \setbox\statebox=\vbox{#1}%
    \edef\thealgruleheight{\dimexpr \the\ht\statebox+1pt\relax}%
    \edef\thealgruledepth{\dimexpr \the\dp\statebox+1pt\relax}%
    \ifdim\thealgruleheight<.75\baselineskip
        \def\thealgruleheight{\dimexpr .75\baselineskip+1pt\relax}%
    \fi
    \ifdim\thealgruledepth<.25\baselineskip
        \def\thealgruledepth{\dimexpr .25\baselineskip+1pt\relax}%
    \fi
    \State #1%
    \def\thealgruleheight{\dimexpr .75\baselineskip+1pt\relax}%
    \def\thealgruledepth{\dimexpr .25\baselineskip+1pt\relax}%
}
\makeatletter
\newcommand{\oset}[3][0ex]{%
  \mathrel{\mathop{#3}\limits^{
    \vbox to#1{\kern-2\ex@
    \hbox{$\scriptstyle#2$}\vss}}}}
\makeatother

\makeatletter
\def\ps@pprintTitle{%
  \let\@oddhead\@empty
  \let\@evenhead\@empty
  \let\@oddfoot\@empty
  \let\@evenfoot\@oddfoot
}
\makeatother
\begin{document}
\begin{frontmatter}
\title{The Numerical Approximation of Caputo Fractional Derivative of Higher-Orders Using A Shifted Gegenbauer Pseudospectral Method: Two-Point Boundary Value Problems of the Bagley–Torvik Type Case Study}
\author[Ajman,NDRC]{Kareem T. Elgindy\corref{cor1}}
\ead{k.elgindy@ajman.ac.ae}
\address[Ajman]{Department of Mathematics and Sciences, College of Humanities and Sciences, Ajman University, P.O.Box: 346 Ajman, United Arab Emirates}
\address[NDRC]{Nonlinear Dynamics Research Center (NDRC), Ajman University, P.O.Box: 346 Ajman, United Arab Emirates}
\cortext[cor1]{Corresponding author}
\begin{abstract}
This work presents a new framework for approximating Caputo fractional derivatives (FDs) of any positive order using a shifted Gegenbauer pseudospectral (SGPS) method. By transforming the Caputo FD into a scaled integral of the $m$th-derivative of the Lagrange interpolating polynomial (with $m$ being the ceiling of the fractional order $\alpha$), we mitigate the singularity near zero, improving stability and accuracy. The method links $m$th-derivatives of shifted Gegenbauer (SG) polynomials with SG polynomials of lower degrees, allowing for precise integration using SG quadratures. We employ orthogonal collocation and SG quadratures in barycentric form to obtain an accurate and efficient approach for solving fractional differential equations. We provide error analysis showing that the SGPS method is convergent in a semi-analytic framework and conditionally convergent with exponential rate for smooth functions in finite-precision arithmetic. This exponential convergence improves accuracy compared to wavelet-based, operational matrix, and finite difference methods. The SGPS method is flexible, with adjustable SG parameters for optimal performance. A key contribution is the fractional SG integration matrix (FSGIM), which enables efficient computation of Caputo FDs via matrix-vector multiplications and accelerates the SGPS method through pre-computation and storage. The method remains within double-precision limits, making it computationally efficient. It handles any positive fractional order $\alpha$ and outperforms existing schemes in solving Caputo fractional two-point boundary value problems (TPBVPs) of the Bagley-Torvik type.
\end{abstract}
\end{frontmatter}

\begin{table*}[h]
\caption{\centering Table of Symbols and Their Meanings}
\centering
\resizebox{\textwidth}{!}{%
\begin{tabular}{|c|c|c|c|c|c|}
\hline
\textbf{Symbol} & \textbf{Meaning} & \textbf{Symbol} & \textbf{Meaning} & \textbf{Symbol} & \textbf{Meaning} \\
\hline
$\forall$ & for all & $\foralla$ & for any & $\forallaa$ & for almost all\\
\hline
$\foralle$ & for each & $\foralls$ & for some & $\forallS$ & for (a) relatively small \\
\hline
$\forallL$ & for (a) relatively large & $\gg$ & much greater than & $\exists$ & there exist(s) \\
\hline
$\sim$ & asymptotically equivalent & $\simlt$ & asymptotically less than & $\simlteq$ & asymptotically less than or equal to \\
\hline
$\not\approx$ & not sufficiently close to & $\MFC$ & set of all complex-valued functions & $\MFF$ & set of all real-valued functions \\
\hline
$\MBC$ & set of complex numbers & $\MBR$ & set of real numbers & $\MBRzer$ & set of non-negative real numbers \\
\hline
$\MBRczer$ & set of nonzero real numbers & $\MBRmhzer$ & $\{x \in \MBR: -1/2 < x < 0\}$ & $\MBZ$ & set of integers \\
\hline
$\MBZP$ & set of positive integers & $\MBZzerP$ & set of non-negative integers & $\MBZeP$ & set of positive even integers \\
\hline
$i$:$j$:$k$ & list of numbers from $i$ to $k$ with increment $j$ & $i$:$k$ & list of numbers from $i$ to $k$ with increment 1 & $y_{1:n}$ or $\left. y_i \right|_{i=1:n}$ & list of symbols $y_1, y_2, \ldots, y_n$ \\
\hline
$\{y_{1:n}\}$ & set of symbols $y_1, y_2, \ldots, y_n$ & $\MBJ_n$ & $\{0:n-1\}$ & $\MBJP_n$ & $\MBJ_n \cup \{n\}$ \\
\hline
$\MBN_n$ & $\{1:n\}$ & $\MBN_{m,n}$ & $\{m:n\}$ & $\MBG_n^{\lambda}$ & set of Gegenbauer-Gauss (GG) zeros of the $(n+1)$st-degree Gegenbauer polynomial with index $\lambda > -1/2$ \\
\hline
$\MBSG_n^{\lambda}$ & set of SGG points in the interval $[0, 1]$ & $\FOmega_{a,b}$ & closed interval $[a, b]$ & $\IFOmega$ & interior of the set $\FOmega$ \\
\hline
$\FOmega_{T}$ & specific interval $[0, T]$ & $\FOmega_{L \times T}$ & Cartesian product $\FOmega_{L} \times \FOmega_{T}$ & $\Gamma(\cdot)$ & Gamma function \\
\hline
$\Gamma(\cdot,\cdot)$ & upper incomplete gamma function & $\left\lceil {.} \right\rceil$ & ceiling function & $\MFI_{j \geq k}$ & indicator (characteristic) function $\begin{cases} 
1 & \text{if } j \geq k, \\
0 & \text{otherwise.}\end{cases}$ \\
\hline
$E_{\alpha, \beta}(z)$ & two-parameter Mittag-Leffler function & $(\cdot)_n$ & Pochhammer symbol & $\supp(f)$ & support of function $f$ \\
\hline
$f^*$ & complex conjugate of $f$ & $f_n$ & $f(t_n)$ & $f_{N,n}$ & $f_N(t_n)$ \\
\hline
$\C{I}_{b}^{(t)}h$ & $\int_0^{b} {h(t)\,dt}$ & $\C{I}_{a, b}^{(t)}h$ & $\int_a^{b} {h(t)\,dt}$ & $\C{I}_t^{(t)}h$ & $\int_0^t {h(.)\,d(.)}$ \\
\hline
$\C{I}_{b}^{(t)}h\cancbra{u(t)}$ & $\int_0^{b} {h(u(t))\,dt}$ & $\C{I}_{a,b}^{(t)}h\cancbra{u(t)}$ & $\int_a^b {h(u(t))\,dt}$ & $\C{I}_{\FOmega_{a,b}}^{(x)} h$ & $\int_a^b {h(x)\,dx}$ \\
\hline
$\Part{x}$ & $d/dx$ & $\NPart{x}{n}$ & $d^n/d x^n$ & $\CapD{x}{\alpha}{f}$ & $\alpha$th-order Caputo FD of $f$ at $x$ \\
\hline
$\Def{\FOmega}$ & space of all functions defined on $\FOmega$ & $C^k(\FOmega)$ & space of $k$ times continuously differentiable functions on $\FOmega$ & $L^p({\FOmega})$ & Banach space of measurable functions $u \in \Def{\FOmega}$ with ${\left\| u \right\|_{{L^p}}} = {\left( {{\C{I}_{\FOmega}}{{\left| u \right|}^p}} \right)^{1/p}} < \infty$ \\
\hline
$L^{\infty}({\FOmega})$ & space of all essentially bounded measurable functions on $\FOmega$ & $\left\|f\right\|_{L^{\infty}(\FOmega)}$ & $L^{\infty}$ norm: $\sup_{x \in \FOmega} |f(x)| = \inf\{M \ge 0: |f(x)| \le M\,\forallaa x \in \FOmega\}$ & $\left\|\cdot\right\|_1$ & $l_1$-norm \\
\hline
$\left\|\cdot\right\|_2$ & Euclidean norm & $\ME{H}^{k,p}(\FOmega)$ & Sobolev space of weakly differentiable functions with integrable weak derivatives up to order $k$ & $\bmt_N$ & $[t_{N,0}, t_{N,1}, \ldots, t_{N,N}]^{\top}$ \\
\hline
$g_{0:N}$ & $[g_0, g_1, \ldots, g_{N}]^{\top}$ & $g^{(0:N)}$ & $[g, g', \ldots, g^{(N)}]^{\top}$ & $c^{0:N}$ & $[1, c, c^2, \ldots, c^{N}]$ \\
\hline
$\bmt_N^{\top}$ or $[t_{N,0:N}]$ & $[t_{N,0}, t_{N,1}, \ldots, t_{N,N}]$ & $h(\bmy)$ & vector with $i$-th element $h(y_i)$ & $\bmh(\bmy)$ or $h_{1:m}\cancbra{\bmy}$ & $[h_1(\bmy), \ldots, h_m(\bmy)]^{\top}$ \\
\hline
$\bmy^{\div}$ & vector of reciprocals of the elements of $\bmy$ & $\F{O}_n$ & zero matrix of size $n$ & $\F{1}_n$ & all ones matrix of size $n$ \\
\hline
$\F{I}_n$ & identity matrix of size $n$ & $\F{C}_{n,m}$ & matrix $\F{C}$ of size $n \times m$ & $\F{C}_n$ & $n$-th row of matrix $\F{C}$ \\
\hline
$\bmone_n$ & $n$-dimensional all ones column vector & $\bmzer_n$ & $n$-dimensional all zeros column vector & $\F{A}^{\top}$ or $\trp{\F{A}}$ & transpose of matrix $\F{A}$ \\
\hline
$\Diag{\bmv}$ & diagonal matrix with $\bmv$ on the diagonal & $\resh{m}{n}{\F{A}}$ & reshape $\F{A}$ into an $m \times n$ matrix & $\reshs{n}{\F{A}}$ & reshape $\F{A}$ into a square matrix of size $n$ \\
\hline
$\otimes$ & Kronecker product & $\odot$ & Hadamard product & $\F{A}_{(r)}$ & $r$-times Hadamard product of $\F{A}$ \\
\hline
$\EE{\F{A}}{m}$ & each entry in $\F{A}$ raised to the power $m$ & & & & \\
\hline
\end{tabular}
}%
\label{tab:symbols}
\caption*{\textit{Remark: A vector is represented in print by a bold italicized symbol while a two-dimensional matrix is represented by a bold symbol, except for a row vector whose elements form a certain row of a matrix where we represent it in bold symbol.}}
\end{table*}
\section{Introduction}
\label{Int}
Fractional calculus extends classical calculus by allowing differentiation and integration of arbitrary real orders, making it a powerful tool for capturing effects like memory, long-range interactions, and anomalous diffusion, which are often observed in various scientific and engineering fields. Among the different definitions of FDs, the Caputo FD is widely used due to its compatibility with classical initial and boundary conditions. Among the numerous mathematical models described by this operator is the Bagley-Torvik equation, which is a well-known fractional differential equation that involves a Caputo FD of order $1.5$. It arises in the modeling of the motion of a rigid plate immersed in a viscous fluid. The FD term represents a damping force that depends on the history of the plate's motion. This type of damping, known as fractional damping or viscoelastic damping, is often used to model materials with memory effects. 

The work on the numerical approximation of FDs and solving the Bagley-Torvik equation is currently a very active research area. In the following, we mention some of the key contributions to the numerical solution of the Bagley-Torvik equation using the Caputo FD: \textbf{Spectral Methods.} \citet{saw2019numerical} proposed a Chebyshev collocation scheme for solving the fractional Bagley-Torvik equation. The Caputo FD was handled through a system of algebraic equations formed using Chebyshev polynomials and specific collocation points. \citet{ji2020numerical} presented a numerical solution using SC polynomials. The Caputo derivative was expressed using an operational matrix of FDs, and the fractional-order differential equation was reduced into a system of algebraic equations that was solved using Newton's method. \citet{hou2020jacobi} solved the Bagley-Torvik equation by converting the differential equation into a Volterra integral equation, which was then solved using Jacobi collocation. \citet{ji2020numericalL} applied Laguerre polynomials to approximate the solution of the Bagley-Torvik equation. The Laplace transform was first used to convert the problem into an algebraic equation, and then Laguerre polynomials were used for numerical inversion. \textbf{Wavelet-Based Methods.} \citet{kaur2019non} developed a hybrid numerical method using non-dyadic wavelets for solving the Bagley-Torvik equation. \citet{dincel2021sine} employed sine-cosine wavelets to approximate the solution of the Bagley-Torvik equation, where the Caputo FD was computed using the operational matrix of fractional integration. \citet{rabiei2022numerical} introduced a wavelet-based technique, utilizing the Riemann-Liouville integral operator to transform the fractional Bagley-Torvik equation into algebraic equations. \textbf{Operational Matrix Methods.} \citet{abd2016spectral} formulated an operational matrix of FDs in the Caputo sense using Lucas polynomials, and applied Tau and collocation methods to solve Bagley-Torvik equation. \citet{youssri2017new} introduced an operational matrix approach using Fermat polynomials for solving the fractional Bagley-Torvik equation in the Caputo sense. A spectral tau method was employed to transform the problem into algebraic equations. \textbf{Galerkin Methods.} \citet{izadi2020local} used a local discontinuous Galerkin scheme with upwind fluxes for solving the Bagley-Torvik equation. The Caputo derivative was approximated by discretizing element-wise systems. \citet{chen2020fast} proposed a fast multiscale Galerkin algorithm using orthogonal functions with vanishing moments. \textbf{Spline and Finite Difference Methods.} \citet{tamilselvan2023second} used a second-order spline approximation for the Caputo FD and a central difference scheme for the second-order derivative term in solving the Bagley-Torvik equation. \textbf{Artificial Intelligence-Based Methods.} \citet{verma2021numerical} employed an artificial neural network method with Legendre polynomials to approximate the solution of the Bagley-Torvik equation, where the Caputo derivative was handled through an optimization-based training process.

This work introduces a novel framework for approximating Caputo FDs of any positive orders using a SGPS method. Unlike traditional approaches, our method employs a strategic change of variables to transform the Caputo FD into a scaled integral of the $m$th-derivative of the Lagrange interpolating polynomial, where $m$ is the ceiling of the fractional order $\alpha$. This transformation mitigates the singularity inherent in the Caputo derivative near zero, thereby improving numerical stability and accuracy. The numerical approximation of the Caputo FD is finally furnished by linking the $m$th-derivative of SG polynomials with another set of SG polynomials of lower degrees and higher parameter values whose integration can be recovered within excellent accuracies using SG quadratures. By employing orthogonal collocation and SG quadratures in barycentric form, we achieve a highly accurate and computationally efficient scheme for solving fractional differential equations under optimal parameter settings. Furthermore, we provide a rigorous error analysis showing that the SGPS method is convergent when implemented within a semi-analytic framework, where all necessary integrals are computed analytically, and is conditionally convergent with an exponential rate of convergence for sufficiently smooth functions when performed using finite-precision arithmetic. This exponential convergence generally leads to superior accuracy compared to existing wavelet-based, operational matrix, and finite difference methods. We conduct a rigorous error and convergence analyses to derive the total truncation error bound of the method and study its asymptotic behavior within double-precision arithmetic. The SGPS is highly flexible in the sense that the SG parameters associated with SG interpolation and quadratures allow for flexibility in adjusting the method to suit different types of problems. These parameters influence the clustering of collocation and quadrature points and can be tuned for optimal performance. A key contribution of this work is the development of the FSGIM. This matrix facilitates the direct computation of Caputo FDs through efficient matrix-vector multiplications. Notably, the FSGIM is constant for a given set of points and parameter values. This allows for pre-computation and storage, significantly accelerating the execution of the SGPS method. The SGPS method avoids the need for extended precision arithmetic, as it remains within the limits of double-precision computations, making it computationally efficient compared to methods that require high-precision arithmetic. The current approach is designed to handle any positive fractional order $\alpha$, making it more flexible than some existing methods that are constrained to specific fractional orders. The efficacy of our approach is demonstrated through its application to Caputo fractional TPBVPs of the Bagley-Torvik type, where it outperforms existing numerical schemes.

The remainder of this paper is structured as follows. Section \ref{sec:SGPS1} introduces the SGPS method, providing a detailed exposition of its theoretical framework and numerical implementation. The computational complexity of the derived FSGIM is discussed in Section \ref{sec:CC20251}. A comprehensive error analysis of the method is carried out in Section \ref{sec:ESA1}, establishing its convergence properties and providing insights into its accuracy. In Section \ref{sec:PS}, we demonstrate the effectiveness of the SGPS method through a case study, focusing on its application to Caputo fractional TPBVPs of the Bagley-Torvik type. Section \ref{sec:CRAC1} presents a series of numerical examples, demonstrating the superior performance of the SGPS method in comparison to existing techniques. Finally, Section \ref{sec:Conc} concludes the paper with a summary of key findings and a discussion of potential future research directions. Tables \ref{tab:symbols} and \ref{tab:acronyms} in \ref{sec:LVA1} display the symbols and acronyms used in the paper and their meanings. \ref{sec:MP1} supports the error analysis conducted in Section \ref{sec:ESA1} by providing rigorous mathematical justifications for the asymptotic order of some key terms in the error bound.

\section{The SGPS Method}
\label{sec:SGPS1}
This section introduces the SGPS method for approximating Caputo fractional derivatives. Readers interested in a deeper understanding of Gegenbauer and SG polynomials, as well as their associated quadratures, are encouraged to consult \cite{Elgindy20161,Elgindy20171,elgindy2018optimal,elgindy2018high}. 

Let $\alpha \in \MBRP\backslash\MBZP, m = \left\lceil  \alpha  \right\rceil, f \in \ME{H}^{m,2}(\FOmega_1), \left\{\hx_{n,0:n}^{\lambda}\right\} = \MBSG_n^{\lambda}$, and consider the following SGPS interpolant of $f$:
\begin{equation}\label{sec:ort:eq:Lagint1}
	{I_n}f(x) = f_{0:n}^{\top}\,\C{L}_{0:n}^{\lambda}\cancbra{x},
\end{equation}
where $\C{L}_{k}^{\lambda}(x)$ is the $n$th-degree Lagrange interpolating polynomial in modal form defined by
\begin{equation}\label{sec:ort:eq:Lag1}
	\C{L}_k^{\lambda}(x) = \hvarpi_k^{\lambda}\,{\trp{{\hlambdabar_{0:n}^{\lambda}}^{\hspace{-1mm} \div}}\,\left(\hG_{0:n}^{\lambda}\cancbra{\hx_{n,k}^{\lambda}} \odot \hG_{0:n}^{\lambda}\cancbra{x}\right)},\quad \forall k \in \MBJP_n,
\end{equation}
${\hlambdabar}_{0:n}^{\lambda}$ and $\hvarpi_{0:n}^{\lambda}$ are the normalization factors for SG polynomials and the Christoffel numbers associated with their quadratures, respectively:
\begin{gather}
\hlambdabar_j^{\lambda} = \frac{{\pi {2^{1 - 4\lambda }}\Gamma \left( {j + 2\lambda } \right)}}{{j!{\Gamma ^2}\left( \lambda  \right)\left( {j + \lambda} \right)}},\\
\hvarpi_k^{\lambda} = 1/\left[{\trp{{\hlambdabar_{0:n}^{\lambda}}^{\hspace{-1mm} \div}}\,\left(\hG_{0:n}^{\lambda}\cancbra{\hx_{n,k}^{\lambda}}\right)_{(2)}}\right],
\end{gather}
$\forall j, k \in \MBJP_n$; cf. \cite[Eqs. (2.6), (2.7), (2.10), and (2.12)]{Elgindy20161}. The matrix form of Eq. \eqref{sec:ort:eq:Lag1} can be stated as:
\begin{equation}\label{eq:matf1}
\C{L}_{0:n}^{\lambda}\cancbra{x} = \Diag{\hvarpi_{0:n}^{\lambda}}\,\left(\hG_{0:n}^{\lambda}\cancbra{x \bmone_{n+1}} \odot \hG_{0:n}^{\lambda}\cancbra{\hbmx_n^{\lambda}}\right)^{\top} {\hlambdabar_{0:n}^{\lambda}}^{\hspace{-1mm} \div}.
\end{equation}
Eq. \eqref{sec:ort:eq:Lagint1} allows us to approximate the Caputo FD of $f$:
\begin{equation}\label{sec:ort:eq:Lagint1_2}
	\CapD{x}{\alpha}{f} \approx \CapD{x}{\alpha}{I_nf} = f_{0:n}^{\top}\,\CapD{x}{\alpha}{\C{L}_{0:n}^{\lambda}}.
\end{equation} 
To accurately evaluate $\CapD{x}{\alpha}{\C{L}_{0:n}^{\lambda}}$, we apply the following\\ $m$-dependent change of variables:
\begin{equation}\label{eq:Elg1}
\tau = x \left(1-y^{\frac{1}{m-\alpha}}\right),
\end{equation}
which magically reduces $\CapD{x}{\alpha}{f}$ into a scalar multiple of the integral of the $m$th-derivative of $f$ on the fixed interval $\FOmega_1$, denoted by $\CapDE{x}{\alpha}{f}$, and defined by
\begin{equation}\label{eq:Henm1}
\CapDE{x}{\alpha}{f} = \frac{x^{m-\alpha}}{\Gamma(m-\alpha+1)} \C{I}_1^{(y)} f^{(m)}\cancbra{x\left(1-y^{\frac{1}{m-\alpha}}\right)}.
\end{equation}
It is easy here to show that the value of $x \left(1 - y^{\frac{1}{m-\alpha}}\right)$ will always lie in the range $\FOmega_x\,\forall\,0 \le x, y \le 1$. Combining Eqs. \eqref{sec:ort:eq:Lagint1_2} and \eqref{eq:Henm1} gives
\begin{equation}\label{eq:Soly1}
\CapD{x}{\alpha}{f} \approx \frac{x^{m-\alpha}}{\Gamma(m-\alpha+1)} f_{0:n}^{\top}\,\C{I}_1^{(y)} {\C{L}_{0:n}^{\lambda, m}}\cancbra{x\left(1- y^{\frac{1}{m-\alpha}}\right)},
\end{equation}
where $\C{L}_j^{\lambda, m}$ denotes the $m$th-derivative of $\C{L}_j^{\lambda}\,\forall j \in \MBJP_n$. Substituting Eq. \eqref{eq:matf1} into Eq. \eqref{eq:Soly1} yields
\begin{gather}
\CapD{x}{\alpha}{f} \approx \frac{x^{m-\alpha}}{\Gamma(m-\alpha+1)}\,\left[\trp{{\hlambdabar_{m:n}^{\lambda}}^{\hspace{-1mm} \div}} \times \right.\\
\left. \left(\C{I}_1^{(y)} {\hG_{m:n}^{\lambda, m}\cancbra{\left(x-x y^{\frac{1}{m-\alpha}}\right) \bmone_{n+1}} \odot \hG_{m:n}^{\lambda}\cancbra{\hbmx_n^{\lambda}}}\right) \Diag{\hvarpi_{0:n}^{\lambda}}\right] f_{0:n},\\\label{eq:Soly1_2}
\end{gather}
where $\hG_j^{\lambda, m}$ denotes the $m$th-derivative of $\hG_j^{\lambda}\,\forall j \in \MBN_{m,n}$. 

To efficiently evaluate Caputo FDs at an arbitrary set of points, $z_{0:M} \in \FOmega_1\,\foralls M \in \MBZzerP$, we can sequentially apply Formula \eqref{eq:Soly1_2} within a loop. While direct implementation using a loop over the vector's elements of $\bmz_M$ is possible, employing matrix operations is highly recommended for substantial performance gains. To this end, notice first that Eq. \eqref{eq:matf1} can be rewritten at $\bmz_M$ as:
\begin{gather}
\C{L}_{0:n}^{\lambda}\cancbra{\bmz_M} = \reshCB{n+1}{M+1}{\trp{{\hlambdabar_{0:n}^{\lambda}}^{\hspace{-1mm} \div}} \times\right.\\
\left. \left(\hG_{0:n}^{\lambda}\cancbra{\bmz_M \otimes \bmone_{n+1}} \odot \hG_{0:n}^{\lambda}\cancbra{\bmone_{M+1} \otimes \hbmx_n^{\lambda}}\right) \times\right.\\
\left. \left(\F{I}_{M+1} \otimes \Diag{\hvarpi_{0:n}^{\lambda}}\right)}.\label{eq:Lagmat1}
\end{gather}
Eq. \eqref{eq:Lagmat1} together with \eqref{eq:Soly1} yield:
\begin{equation}\label{eq:mfkfj1}
\CapD{\bmz_M}{\alpha}{f} \approx \frac{1}{\Gamma(m-\alpha+1)}\,\left[\EE{\bmz_M}{(m-\alpha)} \odot \left(\hCapIM{n}{\alpha}\,f_{0:n}\right)\right],
\end{equation}
where 
\begin{gather}\label{eq:Kimokono1}
\hCapIM{n}{\alpha} = \reshTCB{n+1}{M+1}{\trp{{\hlambdabar_{m:n}^{\lambda}}^{\hspace{-1mm} \div}} \times\right.\\
\left. \left(\C{I}_1^{(y)} {\hG_{m:n}^{\lambda, m}\cancbra{\bmz_M \otimes \left(\left(1-y^{\frac{1}{m-\alpha}}\right) \bmone_{n+1}\right)} \odot \hG_{m:n}^{\lambda}\cancbra{\bmone_{M+1} \otimes \hbmx_n^{\lambda}}}\right) \times\right.\\
\left. \left(\F{I}_{M+1} \otimes \Diag{\hvarpi_{0:n}^{\lambda}}\right)}.
\end{gather}
With a simple algebraic manipulation, we can show further that Eq. \eqref{eq:mfkfj1} can be rewritten as
\begin{equation}\label{eq:mfkfj2}
\CapD{\bmz_M}{\alpha}{f} \approx \CapIM{n}{\alpha}\,f_{0:n},
\end{equation}
where 
\begin{equation}\label{eq:nnnnnn1}
\CapIM{n}{\alpha} = \frac{1}{\Gamma(m-\alpha+1)} \Diag{\EE{\bmz_M}{(m-\alpha)}} \hCapIM{n}{\alpha}.
\end{equation}
We refer to the $(M+1) \times (n+1)$ matrix $\CapIM{n}{\alpha}$ by \textit{``the $\alpha$th-order FSGIM,''} which approximates Caputo FD at the points $z_{0:M}$ using an $n$th-degree SG interpolant. We also refer to $\hCapIM{n}{\alpha}$ by the \textit{``$\alpha$th-order FSGIM Generator''} for an obvious reason. Although the implementation of Formula \eqref{eq:mfkfj2} is straightforward, Formula \eqref{eq:mfkfj1} is more stable and efficient computationally, with fewer arithmetic operations, particularly because it avoids constructing a diagonal matrix and directly applies the elementwise multiplication after the matrix-vector product. Note that for $M = 0$, Formulas \eqref{eq:mfkfj1} and \eqref{eq:mfkfj2} reduce to \eqref{eq:Soly1_2}. 

It remains now to show how to compute 
\[\C{I}_1^{(y)} {\hG_j^{\lambda, m}\cancbra{x\left(1- y^{\frac{1}{m-\alpha}}}\right)},\quad \foralla j \in \MBN_{m:n},\quad x \in \FOmega_1,\]
effectively. Notice first that although the integrand is defined in terms of a polynomial in $x$, the integrand itself is not a polynomial in $y$, since $1/(m-\alpha)$ is not an integer for $\alpha \in \MBRP\backslash\MBZP$. Therefore, when trying to evaluate the integral symbolically, the process can be very challenging and slow. Numerical integration, on the other hand, is often more practical for such integrals because it can achieve any specified accuracy by evaluating the integrand at discrete points without requiring closed-form antiderivatives or algebraic complications. Our faithful tool for this task shall be the SGIM; cf.  \cite{Elgindy20161,elgindy2018optimal,elgindy2025fourier} and the references therein. The SGIM utilizes the barycentric representation of shifted Lagrange interpolating polynomials and their associated barycentric weights to approximate definite integrals effectively through matrix–vector multiplications. The constructed SGPS quadratures by these operations extend the classical Gegenbauer quadrature methods and can improve their performance in terms of convergence speed and numerical stability. An efficient way to construct the SGIM is to premultiply the corresponding GIM by half, rather than shifting the quadrature nodes, weights, and Lagrange polynomials to the target domain $\FOmega_1$, as shown earlier in \cite{Elgindy20161}. In the current work, we only need the GIRV, $\F{P}$, which extends the applicability of the barycentric GIM to include the boundary point $1$; cf. \cite[Algorithm 6 or 7]{Elgindy20171}. The associated SGIRV, $\hFP$, can be directly generated through the formula:
\begin{equation}\label{eq:SG123}
\hFP = \frac{1}{2} \F{P}.
\end{equation}
Given that the construction of $\hFP$ is independent of the SGPS interpolant \eqref{sec:ort:eq:Lagint1}, we can define $\hFP$ using any set of SGG quadrature nodes $\MBSG_{n_q}^{\lambda_q}\,\foralls n_q \in \MBZzerP, \lambda_q > -1/2$. This flexibility enables us to enhance the accuracy of the required integrals without being constrained by the resolution of the interpolation grid. With this strategy, the SGIRV provides a convenient way to approximate the required integral through the following matrix-vector multiplication:
\begin{equation}\label{eq:Solynn1}
\C{I}_1^{(y)} {\hG_j^{\lambda, m}\cancbra{x-x y^{\frac{1}{m-\alpha}}}} \approx \hFP\,\hG_j^{\lambda, m}\left(x\left(1-\EE{\left(\hbmx_{n_q}^{\lambda_q}\right)}{\frac{1}{m-\alpha}}\right)\right),
\end{equation}
$\foralla j \in \MBN_{m:n}, x \in \FOmega_1$. We refer to a quadrature of the form \eqref{eq:Solynn1} by the $(n_q,\lambda_q)$-SGPS quadrature. A remarkable property of Gegenbauer polynomials (and their shifted counterparts) is that their derivatives are essentially other Gegenbauer polynomials, albeit with different degrees and parameters, as shown by the following theorem.

\begin{thm}\label{thm:1}
The $m$th-derivatives of the $n$th-degree, $\lambda$-indexed, Gegenbauer and SG polynomials are given by
\begin{subequations}
\begin{gather}
G_n^{\lambda,m}(x) = \chi_{n,m}^{\lambda} G_{n-m}^{\lambda+m}(x),\label{eq:mnmnmnm12}\\
\hG_n^{\lambda,m}(\hx) = \hchi_{n,m}^{\lambda} \hG_{n-m}^{\lambda+m}(\hx),\label{eq:mnmnmnm13}
\end{gather}
\end{subequations}
where
\begin{gather}
\chi_{n,m}^{\lambda} = \frac{{{2^m}n!\Gamma \left( {2\lambda } \right){{\left( \lambda  \right)}_m}\Gamma \left( {m + n + 2\lambda } \right)}}{{\left( {n - m} \right)!\Gamma \left( {2\left( {m + \lambda } \right)} \right)\Gamma \left( {n + 2\lambda } \right)}},\\
\hchi_{n,m}^{\lambda} = 2^m \chi_{n,m}^{\lambda} = \frac{n! \Gamma(\lambda+1/2) \Gamma(n+m+2 \lambda)}{(n-m)! \Gamma(n+2 \lambda) \Gamma(m+\lambda+1/2)},\\\label{eq:hhkk1}
\end{gather}
$\forall n \ge m, x \in \FOmega_{-1,1}$, and $\hx \in \FOmega_1$.
\end{thm}
\begin{proof}
Let $C_n^{\lambda}(x)$ be the $n$th-degree, $\lambda$-indexed Gegenbauer polynomial  standardized by \citet{Szego1975}. We shall first prove that
\begin{equation}\label{eq:mnmnmnm1}
C_n^{\lambda,m}(x) = 2^m (\lambda)_m C_{n-m}^{\lambda+m}(x),\quad \forall n \ge m,
\end{equation}
where $C_j^{\lambda,m}$ denotes the $m$th-derivative of $C_j^{\lambda}\,\forall j \in \MBN_{m,n}$. To this end, we shall use the well-known derivative formula of this polynomial given by the following recurrence relation:
\[C_n^{\lambda,1}(x) = 2 \lambda C_{n-1}^{\lambda+1}(x),\quad n \ge 1.\]
We will prove Eq. \eqref{eq:mnmnmnm1} by mathematical induction on $m$. The base case $m = 1$ holds true by the given recurrence relation for the first derivative. Assume now that Eq. \eqref{eq:mnmnmnm1} holds true for $m = k$, where $k$ is an arbitrary integer such that $1 < k \le n-1$. That is,
\[C_n^{\lambda,k}(x) = 2^k (\lambda)_k C_{n-k}^{\lambda+k}(x).\]
We need to show that it also holds true for $m=k+1$. Differentiating both sides of the induction hypothesis with respect to $x$ gives:
\begin{gather*}
C_n^{\lambda,k+1}(x) = \frac{d}{dx} \left[ C_n^{\lambda,k}(x) \right] = \frac{d}{dx} \left[ 2^k (\lambda)_k C_{n-k}^{\lambda+k}(x) \right] \\
= 2^k (\lambda)_k \frac{d}{dx} \left[ C_{n-k}^{\lambda+k}(x) \right] = 2^k (\lambda)_k \cdot 2(\lambda+k)\,C_{n-k-1}^{\lambda+k+1}(x) \\
= 2^{k+1} (\lambda)_{k+1} C_{n-k-1}^{\lambda+k+1}(x).
\end{gather*}
This shows that if the formula holds for $m=k$, it also holds for $m=k+1$. By mathematical induction, Eq. \eqref{eq:mnmnmnm1} holds true for all integers $m:0 \le m \le n$. Formula \cite[(A.5)]{Elgindy201382} and the fact that 
\begin{equation}
C_n^{\lambda}(1) = \frac{\Gamma(n+2 \lambda)}{\Gamma(n+1)\,\Gamma(2 \lambda)},
\end{equation}
immediately show that
\begin{equation}\label{eq:mnmnmnm120}
G_n^{\lambda,m}(x) = 2^m (\lambda)_m \frac{C_{n-m}^{\lambda+m}(1)}{C_n^{\lambda}(1)} G_{n-m}^{\lambda+m}(x),\quad \forall n \ge m,
\end{equation}
from which Eq. \eqref{eq:mnmnmnm12} is derived. Formula \eqref{eq:mnmnmnm13} follows from \eqref{eq:mnmnmnm12} by successive application of the Chain Rule.
\end{proof}
Eqs. \eqref{eq:Solynn1} and \eqref{eq:mnmnmnm13} bring to light the sought formula:
\begin{equation}\label{eq:Solynn1nn1}
\C{I}_1^{(y)} {\hG_j^{\lambda, m}\cancbra{x-x y^{\frac{1}{m-\alpha}}}} \approx \hchi_{j,m}^{\lambda}\,\left[\hFP\,\hG_{j-m}^{\lambda+m}\left(x\left(1-\EE{\left(\hbmx_{n_q}^{\lambda_q}\right)}{\frac{1}{m-\alpha}}\right)\right)\right],
\end{equation}
$\foralla j \in \MBN_{m:n}, x \in \FOmega_1$. We denote the approximate $\alpha$th-order Caputo FD of a function at point $x$, computed using Eq. \eqref{eq:Solynn1nn1} in conjunction with either Eq. \eqref{eq:mfkfj1} or Eq. \eqref{eq:mfkfj2}, by $\CapDEN{x}{\alpha}{}$. It is interesting to notice here that the quadrature nodes involved in the computations of the necessary integrals \eqref{eq:Solynn1nn1}, which are required for the construction of the FSGIM $\CapDEN{x}{\alpha}{}$, are independent of the SGG points associated with the SGPS interpolant \eqref{sec:ort:eq:Lagint1}, and therefore, any set of SGG quadrature nodes can be used. This flexibility allows for improving the accuracy of the required integrals without being constrained by the resolution of the interpolation grid.

Figures \ref{fig:1} and \ref{fig:2} show the logarithmic absolute errors of the Caputo FD approximations of $f_1(t) = t^N$ and $f_2(t) = e^{\beta t}: \beta \in \MBRP$ using SG interpolants of various parameters and a $(15,0.5)$-SGPS quadrature. The exact Caputo FD of each function is given below:
\begin{gather*}
\CapD{t}{\alpha}{f_1} = 
\begin{cases} 
\frac{N!}{\Gamma(N+1-\alpha)} t^{N-\alpha}, & N > \alpha - 1, \\
0, & N \leq \alpha - 1,
\end{cases}\quad \forall N \in \MBZzerP, \alpha \in \MBRP,\\
\CapD{t}{\alpha}{f_2} = \mathop \sum \limits_{k = 0}^\infty  \frac{{{\beta ^{k + m}}{t^{ - \alpha  + k + m}}}}{{\Gamma \left( {k + m - \alpha  + 1} \right)}} = \beta^\alpha t^{-\alpha} E_{1, \alpha - m + 1}(\beta t).
\end{gather*}
In all plots of Fig. \ref{fig:1}, we can generally observe the rapid convergence of our PS approximations. For many cases, using an $(N+1)$st-degree Gegenbauer interpolant is often sufficient to approximate the Caputo FD of the power function $t^N$ to within machine precision. However, for certain $\lambda$ values, the error curves sometimes plateau and then exhibit a ``bouncing'' behavior. This plateau and the subsequent fluctuations are due to the accumulation of round-off errors in the PS computation as the approximation approaches ``near'' machine precision. We notice also the near same error profiles for increasing $n$, while fixing $\lambda$. In Fig. \ref{fig:2}, we find that across all parameter values of $\lambda$, the logarithmic absolute errors decrease as the Gegenbauer interpolant's degree increases. This indicates that higher degree interpolants provide more accurate approximations of the Caputo FD up to a certain level of accuracy. For low $n$ values, the errors seem to decrease as $\lambda$ decreases. For larger $n$ values, the overall errors tend to decrease until a certain degree of accuracy regardless of the $\lambda$ value, showing the effectiveness of higher degree interpolants in approximating the Caputo FD accurately. The near-linear error profile in each plot demonstrates the exponential convergence of the PS approximations with rates influenced by the parameter selections as analyzed in Section \ref{sec:ESA1}. 

\begin{figure*}[t]
\centering
\includegraphics[scale=0.275]{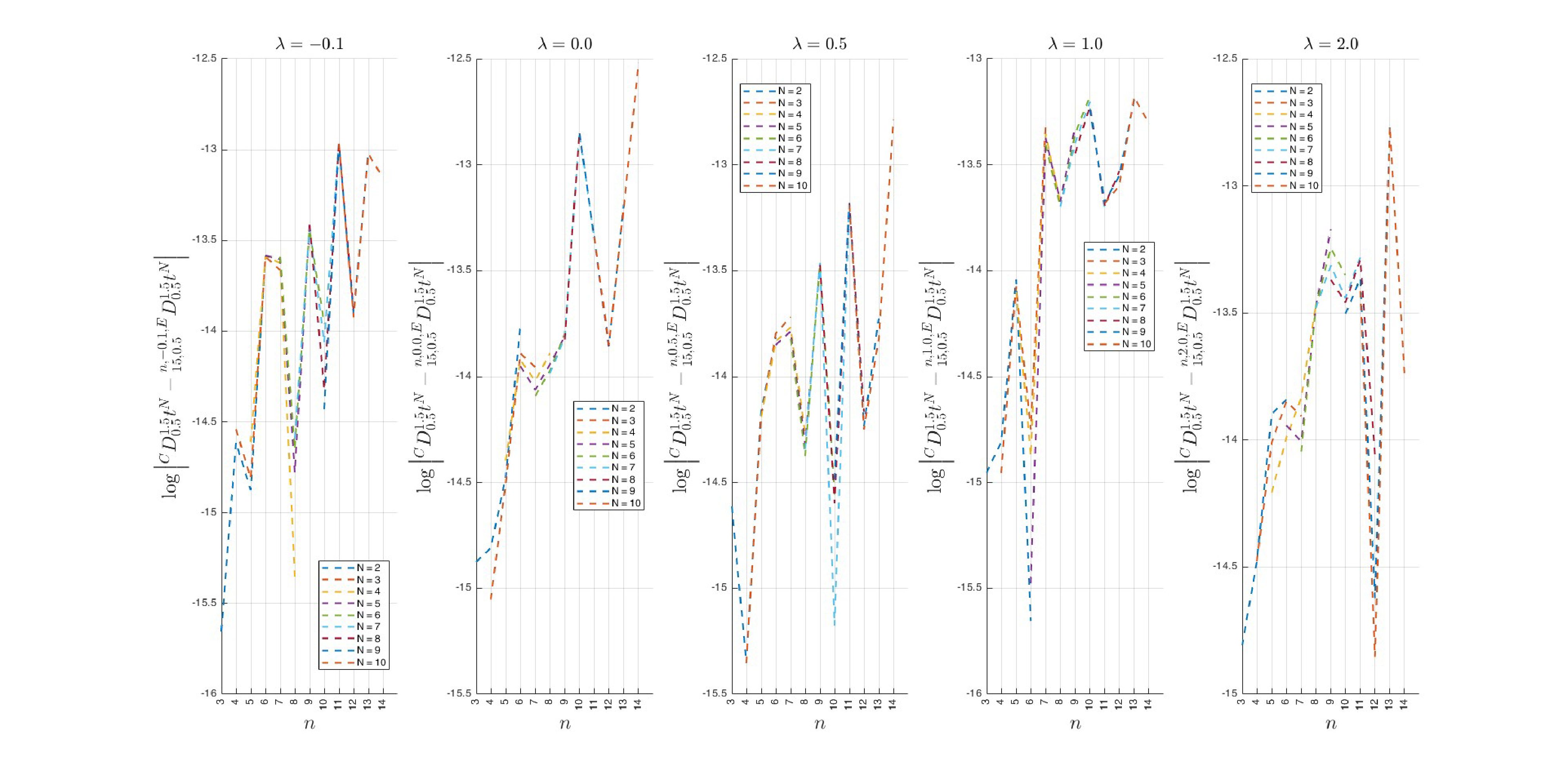}
\caption{Logarithmic absolute errors of the Caputo FD approximations of $f_1$ at $t = 0.5$, for $N = 2:10, \alpha = 1.5$, using Gegenbauer interpolants of degrees $n = N+1:N+4$ and parameter values $\lambda = -0.1, 0, 0.5, 1, 2$, together with a $(15,0.5)$-SGPS quadrature. Each curve within a plot corresponds to a different power $N$ of the function $t^N$, ranging from $2$ to $10$.}
\label{fig:1}
\end{figure*}

\begin{figure*}[t]
\centering
\includegraphics[scale=0.275]{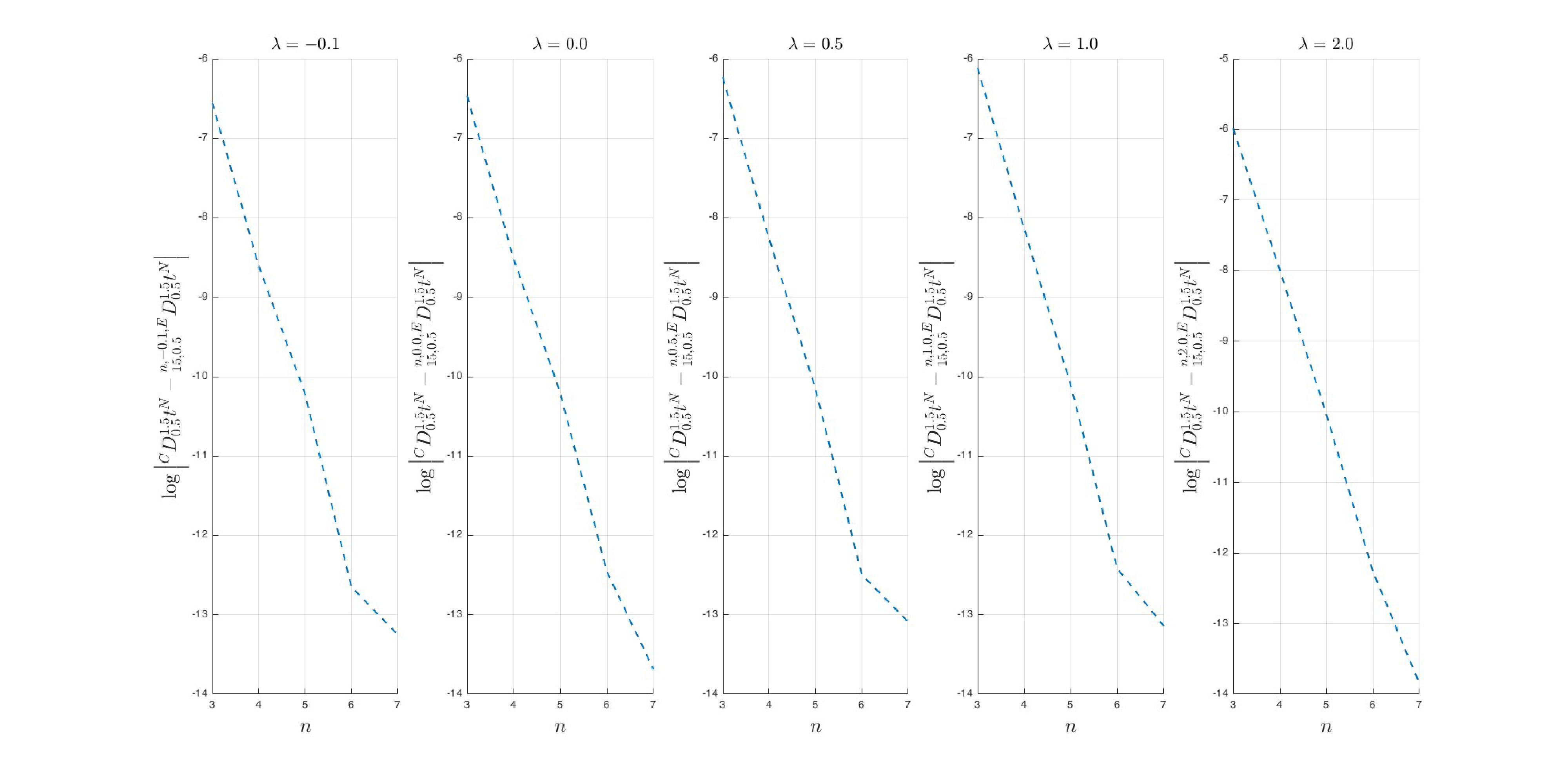}
\caption{Logarithmic absolute errors of the Caputo FD approximations of $f_2$ at $t = 0.5$, for $\beta = 0.1, \alpha = 1.5$, using Gegenbauer interpolants of degrees $n = 3:7$ and parameter values $\lambda = -0.1, 0, 0.5, 1, 2$, together with a $(15,0.5)$-SGPS quadrature.}
\label{fig:2}
\end{figure*}

\section{Computational Complexity}
\label{sec:CC20251}
In this section, we provide a computational complexity analysis of constructing $\CapIM{n}{\alpha}$, incorporating the quadrature approximation \eqref{eq:Solynn1nn1}. The analysis is based on the key matrix operations involved in the construction process, which we analyze individually in the following: Observe from Eq. \eqref{eq:nnnnnn1} that the term \(\EE{\bmz_M}{(m-\alpha)}\) involves raising each element of an \((M+1)\)-dimensional vector to the power \((m-\alpha)\), which requires \(O(M)\) operations. Constructing \(\CapIM{n}{\alpha}\) from \(\hCapIM{n}{\alpha}\) involves a diagonal scaling by \(\Diag{\EE{\bmz_M}{(m-\alpha)}}\), which requires another \(O(Mn)\) operations. The matrix \(\hCapIM{n}{\alpha}\) is constructed using several matrix multiplications and elementwise operations. For each entry of $\bmz_M$, the dominant steps include: 
\begin{itemize}
\item The computation of \(\hG_{m:n}^{\lambda}\). Using the three-term recurrence equation:
\begin{equation}
(n + 2  \alpha )  \hG_{n + 1}^{(\lambda)}(\hx) = 2  (n + \alpha )  (2\hx-1) \hG_n^{(\lambda)}(\hx) - n \hG_{n - 1}^{(\lambda)}(\hx), \end{equation}
$\forall n \in \MBZP$, starting with $\hG_0^{(\lambda)}(\hx) = 1$ and $\hG_1^{(\lambda)}(\hx) = 2 \hx - 1$, we find that each polynomial evaluation requires $O(1)$ per point, as the number of operations remains constant regardless of the value of $n$. Since the polynomial evaluation is required for polynomials up to degree $n$, this requires $O(n)$ operations per point. The computations of $\hG_{m:n}^{\lambda}\cancbra{\hbmx_n^{\lambda}}$ therefore require $O(n^2)$ operations.
\item The quadrature \eqref{eq:Solynn1nn1} involves evaluating a polynomial at transformed nodes. The cost of calculating $\hchi_{j,m}^{\lambda}$ depends on the chosen methods for computing factorials and the Gamma function. It can be considered a constant overhead for each evaluation of the Eq. \eqref{eq:hhkk1}. The computation of $\EE{\left(\hbmx_{n_q}^{\lambda_q}\right)}{\frac{1}{m-\alpha}}$ involves raising each element of the column vector $\hbmx_{n_q}^{\lambda_q}$ to the power $1/(m-\alpha)$. The cost here is linear in $(n_q+1)$, as each element requires a single exponentiation operation. Since we need to evaluate the polynomial at $n_q+1$ points, the total cost for this step is $O(n_q)$. The cost of the matrix-vector multiplication is also linear in $n_q+1$. Therefore, the computational cost of this step is $O(n_q)$ for each $j \in \MBN_{m:n}$. The overall cost, if we consider all polynomial functions involved in this step, is thus $O(n n_q)$.
\item The Hadamard product introduces another \(O(n^2)\) operations.
\item The evaluation of ${\hlambdabar_{m:n}^{\lambda}}^{\hspace{-1mm} \div}$ requires $O(n)$ operations, and the product of $\trp{{\hlambdabar_{m:n}^{\lambda}}^{\hspace{-1mm} \div}}$ by the result from the Hadamard product requires $O(n^2)$ operations.
\item The final diagonal scaling \(\Diag{\hvarpi_{0:n}^{\lambda}}\) contributes \(O(n)\).
\end{itemize}
Summing the dominant terms, the overall computational complexity of constructing \(\CapIM{n}{\alpha}\) is of $O\left(n (n + n_q)\right)$ per entry of $\bmz_M$. We therefore expect the total number of operations required to construct the matrix \(\CapIM{n}{\alpha}\) for all entries of $\bmz_M$ to be of $O\left(M n (n + n_q)\right)$.

\section{Error Analysis}
\label{sec:ESA1}
The following theorem defines the truncation error of the $\alpha$th-order SGPS quadrature \eqref{eq:mfkfj2} associated with the $\alpha$th-order FSGIM $\CapIM{n}{\alpha}$ in closed form.

\begin{thm}\label{subsec:err:thm1}
Let $n \ge m-1$ and suppose that $f \in C^{n + 1}(\FOmega_1)$ is approximated by the SGPS interpolant \eqref{sec:ort:eq:Lagint1}. Assume also that the integrals
\begin{equation}\label{eq:sdmkcmlvh1}
\C{I}_1^{(y)} {\hG_{m:n}^{\lambda, m}\cancbra{x\left(1-y^{\frac{1}{m-\alpha}}\right)}},
\end{equation}
are computed exactly $\foralla x \in \FOmega_1$. Then $\exists\,\xi = \xi(x) \in \IFOmega_1$ such that the truncation error, ${}^{\alpha}\ME{T}_{n}^{\lambda}(x, \xi)$, in the Caputo FD approximation \eqref{eq:Soly1_2} is given by
\begin{equation}\label{eq:Clev121}
{}^{\alpha}\ME{T}_{n}^{\lambda}(x, \xi) = {}^{\alpha}\eta_n^{\lambda}\,f^{(n+1)}(\xi)\,\C{I}_x^{(\tau)}{\frac{\hG_{n+1-m}^{\lambda+m}}{(x-\tau)^{\alpha+1-m}}},
\end{equation}
where 
\begin{equation}
{}^{\alpha}\eta_n^{\lambda} = \frac{{\sqrt \pi  {2^{ - 2\lambda  - 2n - 1}}\Gamma \left( {m + n + 2\lambda  + 1} \right)}}{(n - m + 1)!\,{\Gamma \left( {m - \alpha } \right)\Gamma \left( {m + \lambda  + \frac{1}{2}} \right) \Gamma \left( {n + \lambda  + 1} \right)}}.
\end{equation}
\end{thm}
\begin{proof}
The Lagrange interpolation error associated with the SGPS interpolation \eqref{sec:ort:eq:Lagint1} is given 
\begin{equation}
	f(x) = {I_n}f(x) + \frac{{{f^{(n + 1)}}(\xi )}}{{(n + 1)!\,\hK_{n + 1}^{\lambda}}}\,\hG_{n + 1}^{\lambda}(x),
\end{equation}
where $\hK_n^{\lambda}$ is the leading coefficient of the $n$th-degree, $\lambda$-indexed SG polynomial; cf. \cite[Eq. (4.12)]{Elgindy20161}. Applying Caputo FD on both sides of the equation gives the truncation error associated with Formula \eqref{eq:Soly1_2} in the following form:
\begin{gather}
{}^{\alpha}\ME{T}_{n}^{\lambda}(x, \xi) = \frac{{{f^{(n + 1)}}(\xi )}}{{(n + 1)!\,\hK_{n + 1}^{\lambda}}}\,\CapD{x}{\alpha}{\hG_{n + 1}^{(\lambda)}}\\
 = \frac{{{f^{(n + 1)}}(\xi )}}{{(n + 1)!\,\hK_{n + 1}^{\lambda}} \Gamma(m-\alpha)}\,\C{I}_x^{(\tau)} \frac{\hG_{n + 1}^{(\lambda, m)}}{(x-\tau)^{\alpha+1-m}}.\label{eq:mcbvs1}
\end{gather}
The proof is established by substituting Formula \eqref{eq:mnmnmnm13} into \eqref{eq:mcbvs1}.
\end{proof}

The following theorem marks the truncation error bound associated with Theorem \ref{subsec:err:thm1}.

\begin{thm}\label{thm:csbhvs12}
Suppose that the assumptions of Theorem \ref{subsec:err:thm1} hold true. Then the truncation error ${}^{\alpha}\ME{T}_{n}^{\lambda}(x, \xi)$ is asymptotically bounded above by
\begin{equation}\label{eq:Asymerrbdv1}
\left|{}^{\alpha}\ME{T}_{n}^{\lambda}(x, \xi)\right| \simlt\;A {\hvartheta _{m ,\lambda }}{2^{ - 2\lambda  - 2n - \frac{3}{2}}}{n^{\lambda  + m}}\quad \forallL n, 
\end{equation}
where $A = \left\|f^{(n+1)}\right\|_{L^{\infty}(\FOmega_1)}$ and 
\begin{gather}
{\hvartheta _{m,\lambda }} = \frac{1}{{\sqrt {e} }}{\left( {\lambda  + m - \frac{1}{2}} \right)^{ - \lambda  - m}} \times\\
{\left( {\left( {\lambda  + m - \frac{1}{2}} \right)\sinh \left( {\frac{1}{{\lambda  + m - \frac{1}{2}}}} \right)} \right)^{\frac{1}{4}\left( { - 2\lambda  - 2m + 1} \right)}}.
\end{gather}
\end{thm}
\begin{proof}
Since $\lambda+m > 3/2 > 0$, \cite[Eq. (4.29a)]{Elgindy20161} shows that $\left\|\hG_{n+1-m}^{\lambda+m}\right\|_{L^{\infty}(\FOmega_1)} = 1$. Thus,
\begin{equation}\label{eq:Bestinq1}
\left|\C{I}_x^{(\tau)}{\frac{\hG_{n+1-m}^{\lambda+m}}{(x-\tau)^{\alpha+1-m}}}\right| \le \C{I}_x^{(\tau)}{(x-\tau)^{m-\alpha-1}} = \frac{x^{m-\alpha}}{m-\alpha} \le \frac{1}{m-\alpha},
\end{equation}
by the Mean Value Theorem for Integrals. Notice also that $\Gamma(z) > 1/z\,\forall z \in \IFOmega_1$. Combining this elementary inequality with the sharp inequalities of the Gamma function \cite[Inequality (96)]{elgindy2018optimal} imply
\begin{gather}
\left|{}^{\alpha}\eta_n^{\lambda}\right| < \frac{1}{{\sqrt {e} }}\left( {m - \alpha } \right){\left( {\lambda  + m - \frac{1}{2}} \right)^{ - \lambda  - m}}{2^{ - 2\lambda  - 2n - \frac{3}{2}}}{\left( {\lambda  + n} \right)^{ - \lambda  - n - \frac{1}{2}}} \times\\
{\left( {\left( {\lambda  + m - \frac{1}{2}} \right)\sinh \left( {\frac{1}{{\lambda  + m - \frac{1}{2}}}} \right)} \right)^{\frac{1}{4}\left( { - 2\lambda  - 2m + 1} \right)}}{\left( {2\lambda  + m + n} \right)^{2\lambda  + m + n + \frac{1}{2}}} \times\\
\left( {\frac{1}{{1620{{\left( {2\lambda  + m + n} \right)}^5}}} + 1} \right){\left( {\left( {\lambda  + n} \right)\sinh \left( {\frac{1}{{\lambda  + n}}} \right)} \right)^{\frac{1}{2}\left( { - \lambda  - n} \right)}} \times\\
{\left( {\left( {2\lambda  + m + n} \right)\sinh \left( {\frac{1}{{2\lambda  + m + n}}} \right)} \right)^{\lambda  + \frac{{m + n}}{2}}} \sim {\vartheta _{\alpha ,\lambda }}{2^{ - 2\lambda  - 2n - \frac{3}{2}}}{n^{\lambda  + m}}\; \forallL n,\label{eq:sauydff1}
\end{gather}
where ${\vartheta _{\alpha ,\lambda }} = (m-\alpha) {\hvartheta _{m ,\lambda }}$. 
The required asymptotic formula \eqref{eq:Asymerrbdv1} is derived by combining the asymptotic formula \eqref{eq:sauydff1} with Ineq. \eqref{eq:Bestinq1}.
\end{proof}
Since the dominant term in the asymptotic bound \eqref{eq:Asymerrbdv1} is ${2^{ - 2\lambda  - 2n - \frac{3}{2}}}$, the truncation error exhibits an exponential decay as $n \to \infty$. Notice also that increasing $\alpha$, while holding $\lambda$ fixed and keeping $n$ sufficiently large, leads to an increase in $m$, which in turn affects two factors: (i) The polynomial term $n^{\lambda + m}$ grows, which slightly slows convergence, and (ii) the prefactor ${\hvartheta _{m ,\lambda }} \sim e^{-1/2}\,m^{-(\lambda+m)}\,\forallL m$, which decreases exponentially, reducing the error; cf. Figure \ref{fig:vartheta1}. Despite the polynomial growth of the former factor, the exponential decay term \( 2^{-2n} \) dominates. Now, let us consider the effect of changing $\lambda$, while holding $\alpha$ fixed and $n$ large enough. If we increase $\lambda$ gradually, the term \( 2^{-2\lambda} \) would exhibit an exponential decay and the prefactor ${\hvartheta _{m ,\lambda }} \sim e^{-1/2} \lambda^{-(\lambda+m)}\,\forallL \lambda$ will also decrease exponentially, further reducing the error. The polynomial term $n^{\lambda + m}$, on the other hand, will increase, slightly increasing the error. 
Although the polynomial term $n^{\lambda + m}$ grows and slightly increases the error, the dominant exponential decay effects from both $2^{-2\lambda}$ and the prefactor $\hvartheta_{m,\lambda}$ ensure that the truncation error decreases significantly as $\lambda$ increases. Hence, increasing $\lambda$ leads to a faster decay of the truncation error. This analysis shows that for $\forallL n$, increasing $\alpha$ slightly increases the error bound due to polynomial growth but does not affect exponential convergence. Furthermore, increasing $\lambda$ generally improves convergence since the exponential decay dominates the polynomial growth. In fact, one can see this last remark from other two viewpoints:
\begin{enumerate}[label=(\roman*)]
\item $\forallL n/(m-1), \supp\left(\hG_{n+1-m}^{\lambda+m}\right) \to \{0, 1\}\,\forallL \lambda$, and the truncation error ${}^{\alpha}\ME{T}_{n}^{\lambda} \to 0$ accordingly.
\item $\forall \lambda \in \MBRP, \supp\left(\hG_j^{\lambda, m}\right) \to \{0, 1\}$, as $j/m \to \infty$. Consequently, the integrals \eqref{eq:sdmkcmlvh1} collapse $\forall \hG_k^{\lambda, m}: m < k \le n,\, k \gg m$, indicating faster convergence rates in the Caputo FD approximation \eqref{eq:Soly1_2}. 
\end{enumerate}
In all cases, choosing a sufficiently large $n$ ensures overall exponential convergence. It is important to note that that these observations are based on the asymptotic behavior of the error upper bound as $n \to \infty$, assuming the SGPS quadrature is computed exactly. $\forallS n$ values, any $\lambda$ value within the recommended range $[-1/2+\varepsilon, 2]$ might be optimal.

Beyond the convergence considerations mentioned above, we highlight two important numerical stability issues related to this analysis: 
\begin{enumerate}[label=(\roman*)]
\item A small buffer parameter $\varepsilon$ is often introduced to offset the instability of the SG interpolation near $\lambda = -1/2$, where SG polynomials grow rapidly for increasing orders \cite{Elgindy20161}. 
\item As $\lambda$ increases, the SGG nodes $\hx_{n,0:n}^{\lambda}$ cluster more toward the center of the interval. This means that the SGPS interpolation rule \eqref{sec:ort:eq:Lagint1} relies more on extrapolation rather than interpolation, making it more sensitive to perturbations in the function values and amplifying numerical errors. This consideration reveal that, although increasing $\lambda$ theoretically improves the convergence rate, it can introduce numerical instability due to increased extrapolation effects. Therefore, when selecting $\lambda$, one must balance convergence speed against numerical stability considerations to ensure accurate interpolation computations. This aligns well with the widely accepted understanding that, for sufficiently smooth functions and sufficiently large spectral expansion terms, the truncated expansion in the SC quadrature (corresponding to $\lambda = 0$) is optimal in the $L^{\infty}$-norm for definite integral approximations; cf. \cite{Elgindy201382} and the references therein. 
\end{enumerate}
In the following, we study the truncation error of the quadrature formula \eqref{eq:Solynn1nn1} and how does its outcomes add up to the above analysis.

\begin{figure}[ht]
\centering
\includegraphics[scale=0.45]{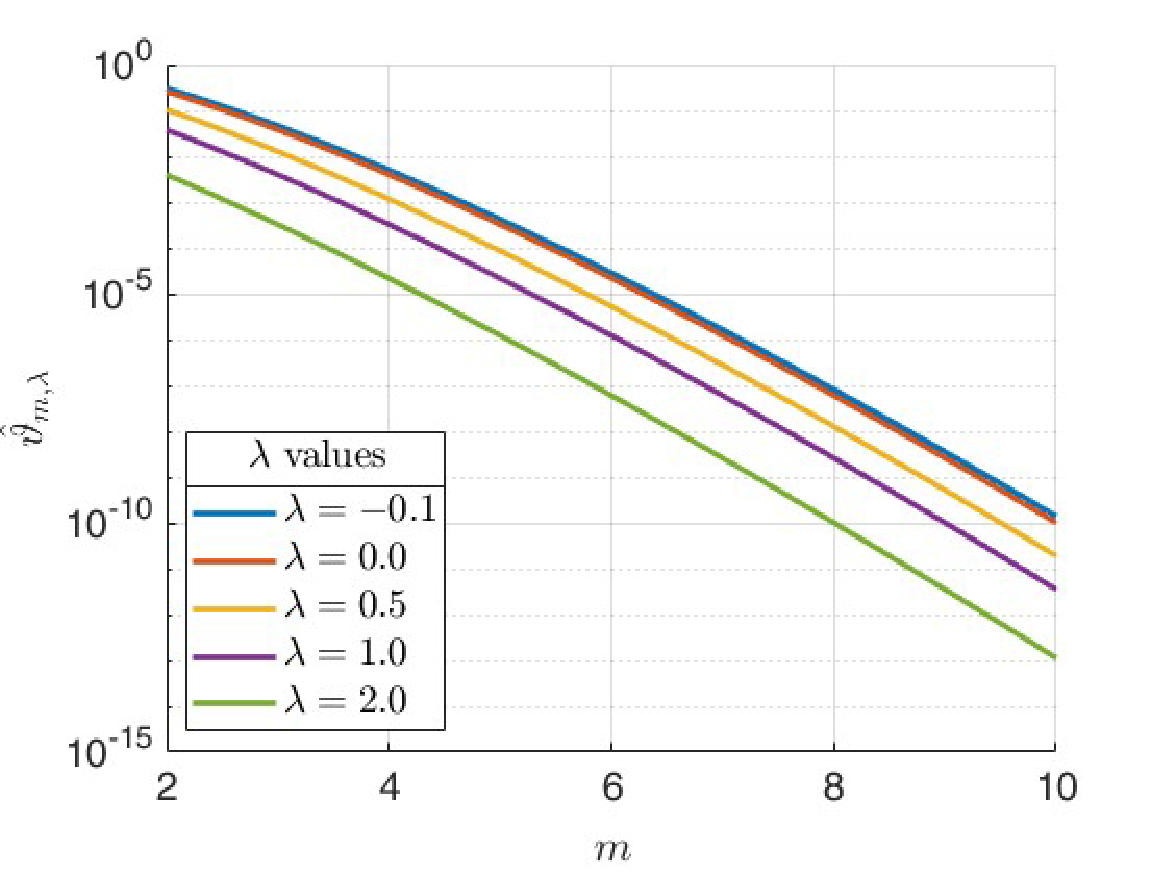}
\caption{Log-lin plots of ${\hvartheta _{m,\lambda }}$ for $\lambda = -0.1, 0, 0.5, 1, 2$, and $m = 2:10$.}
\label{fig:vartheta1}
\end{figure}

\begin{thm}\label{subsec:err:thm3}
Let $j \in \MBN_{m:n}, x \in \FOmega_1$, and assume that\\ $\hG_{j-m}^{\lambda+m}\left(x-x y^{\frac{1}{m-\alpha}}\right)$ is interpolated by the SG polynomials with respect to the variable $y$ at the SGG nodes $\hx_{n_q,0:n_q}^{\lambda_q}$. Then $\exists\,\eta = \eta(y) \in \IFOmega_1$ such that the truncation error, $\FR{T}_{j, n_q}^{\lambda_q}(\eta)$, in the quadrature approximation \eqref{eq:Solynn1nn1} is given by
\begin{gather}
\FR{T}_{j, n_q}^{\lambda_q}(\eta) = \frac{(-1)^{n_q+1} \hchi_{j-m,n_q+1}^{\lambda+m}}{(n_q+1)! \hK_{n_q+1}^{\lambda_q}} \left(\frac{x}{m-\alpha}\right)^{n_q+1} \eta^{\frac{(n_q+1) (1-m+\alpha)}{m-\alpha}} \times\\
 \hG_{j-m-n_q-1}^{\lambda+m+n_q+1}\left(x-x \eta^{\frac{1}{m-\alpha}}\right)\,\C{I}_1^{(y)} {\hG_{n_q+1}^{\lambda_q}} \cdot \MFI_{j \ge m+n_q+1}.\label{subsec:err:eq:squadki1}
\end{gather}
\end{thm}
\begin{proof}
\cite[Theorem 4.1]{Elgindy20161} immediately shows that
\begin{gather}
\FR{T}_{j, n_q}^{\lambda_q}(\eta) = \frac{1}{(n_q+1)! \hK_{n_q+1}^{\lambda_q}} \left[\NPart{y}{n_q+1} \hG_{j-m}^{\lambda+m}\left(x-x y^{\frac{1}{m-\alpha}}\right)\right]_{y=\eta} \C{I}_1^{(y)} {\hG_{n_q+1}^{\lambda_q}}\label{eq:Pr1oof1}\\
 = \frac{(-1)^{n_q+1}}{(n_q+1)! \hK_{n_q+1}^{\lambda_q}} \left(\frac{x}{m-\alpha}\right)^{n_q+1} \eta^{\frac{(n_q+1) (1-m+\alpha)}{m-\alpha}} \times\\
 \hG_{j-m}^{\lambda+m,n_q+1}\left(x-x \eta^{\frac{1}{m-\alpha}}\right)\,\C{I}_1^{(y)} {\hG_{n_q+1}^{\lambda_q}},\label{eq:Pr1oof2}
\end{gather}
by the Chain Rule. The error bound \eqref{subsec:err:eq:squadki1} is accomplished by susbtituting Formula \eqref{eq:mnmnmnm13} into \eqref{eq:Pr1oof2}. The proof is completed by realizing further that
\[\hG_{j-m}^{\lambda+m,n_q+1}\left(x-x \eta^{\frac{1}{m-\alpha}}\right) = \left.\NPart{\tau}{n_q+1} \hG_{j-m}^{\lambda+m}\left(\tau\right)\right|_{\tau = x-x \eta^{\frac{1}{m-\alpha}}} = 0,\]
$\forall j < m+n_q+1$.
\end{proof}

The truncation error analysis of the quadrature approximation \eqref{eq:Solynn1nn1} hinges on understanding the interplay between the parameters \( j, n_q, m, \lambda \), and \( \lambda_q \). While Theorem \ref{subsec:err:thm3} provides an exact expression for the error, the next theorem establishes a rigorous asymptotic upper bound, revealing how the error scales with these parameters.

\begin{thm}\label{subsec:err:thm2}
Let the assumptions of Theorem \ref{subsec:err:thm3} hold true. Then the truncation error, $\FR{T}_{j, n_q}^{\lambda_q}(\eta)$, in the quadrature approximation \eqref{eq:Solynn1nn1} is bounded above by:
\begin{gather}
\left|\FR{T}_{j, n_q}^{\lambda_q}(\eta)\right| \simlt B_m^{\lambda ,{\lambda _q}}{2^{- 2{n_q}}} {\left( {j - m - {n_q}} \right)^{ - j + m + {n_q} + \frac{1}{2}}} {j^{ - 2\lambda  - 2m + 1}} \times\\
{\left( {j + {n_q}} \right)^{j + 2\lambda  + m + {n_q} + \frac{1}{2}}}n_q^{- 2n_q  - m - \lambda  + \lambda_q - \frac{5}{2}} \left(\frac{x}{m-\alpha}\right)^{n_q+1} \eta^{\frac{(n_q+1) (1-m+\alpha)}{m-\alpha}} \times\\
\Upsilon_{D}^{\lambda_q}(n_q)\,\MFI_{j \ge m+n_q+1},\\
\label{eq:SolySofy20251}
\end{gather}
$\forallL n_q$, where 
\[\Upsilon_{D}^{\lambda_q}(n_q) = \begin{cases}
1,\quad \lambda_q \in \MBRzerP, \\
D^{\lambda_q} n_q^{-\lambda_q},\quad \lambda_q \in \MBRmhzer,\\
\end{cases}\]
$D^{\lambda_q} > 1$ is a constant dependent on ${\lambda_q}$, and $B_m^{\lambda ,{\lambda _q}}$ is a constant dependent on $m, \lambda$, and $\lambda _q$.
\end{thm}
\begin{proof}
\cite[Lemma 5.1]{elgindy2018high} shows that 
\begin{equation}\label{eq:sgvcmfhm1}
 \left\| {\hG_k^{\gamma}} \right\|_{L^{\infty}(\FOmega_1)} =
\begin{cases}
1,\quad k \in \MBZzerP,\;  \gamma \in \MBRzerP, \\
\sigma^{\gamma} k^{-\gamma}, \quad \gamma \in \MBRmhzer,\; k \rightarrow \infty,\\
\end{cases}
\end{equation}
where $\sigma^{\gamma} > 1$ is a constant dependent on $\gamma$. Therefore, 
\[\left|\hG_{j-m-n_q-1}^{\lambda+m+n_q+1}\left(x-x \eta^{\frac{1}{m-\alpha}}\right)\right| \le 1,\]
since $\lambda+m+n_q+1 > 0$. Moreover, Formula \eqref{eq:hhkk1} and the definition of $\hK_{n_q}^{\lambda_q}$ \cite[see p.g. 103]{elgindy2018high} show that
\begin{gather}
\frac{\hchi_{j-m,n_q+1}^{\lambda+m}}{(n_q+1)! \hK_{n_q+1}^{\lambda_q}} = \frac{{{2^{ - 2{n_q} - 1}}\Gamma \left( {{\lambda _q} + 1} \right)\left( {j - m} \right)!}}{{\Gamma \left( {2{\lambda _q} + 1} \right)\Gamma \left( {{n_q} + 2} \right)\Gamma \left( {{n_q} + {\lambda _q} + 1} \right)}} \times\\
\frac{{\Gamma \left( {m + \lambda  + \frac{1}{2}} \right)\Gamma \left( {{n_q} + 2{\lambda _q} + 1} \right)\Gamma \left( {j + m + {n_q} + 2\lambda  + 1} \right)}}{{\Gamma \left( {j + m + 2\lambda } \right)\Gamma \left( {j - m - {n_q}} \right)\Gamma \left( {m + {n_q} + \lambda  + \frac{3}{2}} \right)}}.\\\label{eq:bnmbn1}
\end{gather}
The proof is established by applying the sharp inequalities of the Gamma function \cite[Inequality (96)]{elgindy2018optimal} on \eqref{eq:bnmbn1}.
\end{proof}

When $m \ll n_q$, the analysis of Theorem \ref{subsec:err:thm2} bifurcates into the following two essential cases:
\begin{description}
\item[Case I ($j \sim n_q$)]: Let $j = m+n_q+k+1: k = O(1), k \ll n_q$. The first few error factors in \eqref{eq:SolySofy20251} can be simplified as follows:
\begin{align*}
&{2^{- 2{n_q}}} {\left( {j - m - {n_q}} \right)^{ - j + m + {n_q} + \frac{1}{2}}} {j^{ - 2\lambda  - 2m + 1}} {\left( {j + {n_q}} \right)^{j + 2\lambda  + m + {n_q} + \frac{1}{2}}} \times\\
&n_q^{- 2n_q  - m - \lambda  + \lambda_q - \frac{5}{2}} \sim {2^{- 2{n_q}}} {\left( k+1 \right)^{-k-1/2}} {n_q^{- 2\lambda - 2m +1}} {\left( {2{n_q}} \right)^{2n_q + 2\lambda + m +\frac{1}{2}}} \times\\
&n_q^{- 2{n_q} - m - \lambda  + {\lambda _q} - \frac{5}{2}} \sim {2^{\frac{1}{2} + m + 2\lambda }}{\left( {k + 1} \right)^{ - k - \frac{1}{2}}}n_q^{ - 1 - 2m - \lambda  + {\lambda _q}}.
\end{align*}
The dominant exponential decay factor in $\sup \left|\FR{T}_{j, n_q}^{\lambda_q}(\eta)\right|$ is therefore:
\begin{equation}\label{eq:cscshgcvsh1}
\Lambda_{n_q,m}^{\alpha}(x,\eta) = \left(\frac{x}{m - \alpha}\right)^{n_q + 1} \eta^{\frac{(n_q + 1)(1 - m + \alpha)}{m - \alpha}}.
\end{equation}
This shows that the error bound decays exponentially with $n_q$ if 
\begin{equation}\label{eq:convcondver1}
\frac{x\,\eta^{\frac{1 - m + \alpha}{m - \alpha}}}{m - \alpha} < 1,
\end{equation}
is satisfied. Observe that increasing $\lambda$ accelerates the algebraic decay, driven by the polynomial term $n_q^{-1-2m-\lambda+\lambda_q}$. While increasing $\lambda_q$ can counteract this acceleration, the exponential term eventually dictates the convergence rate. Practically, to improve the algebraic decay in this case, we can increase $\lambda$ and choose $\lambda_q: \lambda_q \le \lambda+2m+1$ to prevent the polynomial term growth.
\item[Case II ($j \gg n_q$)]: By Lemma \ref{lem:hello1}, the dominant terms involving $j$ are approximately ${e^{\frac{{2\lambda {n_q}}}{j}}} j^{2 n_q + 2} \approx j^{2 n_q + 2}$. This result can also be derived from the asymptotic error bound in \eqref{eq:SolySofy20251} by observing that $j - n_q \sim j \sim j + n_q$. Thus, the dominant terms involving $j$,
\[{\left( {j - {n_q}} \right)^{- j + m + {n_q}+\frac{1}{2}}} {j^{- 2\lambda - 2m + 1}} {\left( {j + {n_q}} \right)^{j + 2\lambda + m + {n_q} + \frac{1}{2}}},
\]
reduce to approximately $j^{2n_q+2}$. Consequently, the error bound becomes
\begin{align*}
\left|\FR{T}_{j, n_q}^{\lambda_q}(\eta)\right| &\simlt B_m^{\lambda, \lambda_q} 2^{- 2n_q} j^{2n_q+2} n_q^{- 2n_q - m - \lambda + \lambda_q - \frac{5}{2}} \times\\
&\quad \left(\frac{x}{m - \alpha}\right)^{n_q + 1} \eta^{\frac{(n_q + 1)(1 - m + \alpha)}{m - \alpha}}\,\Upsilon_D^{\lambda_q}(n_q)\\
&= B_m^{\lambda, \lambda_q} \left(\frac{j^2}{4 n_q^2}\right)^{n_q} \frac{j^2}{n_q^{m + \lambda - \lambda_q + \frac{5}{2}}} \times\\
&\quad \left(\frac{x}{m - \alpha}\right)^{n_q + 1} \eta^{\frac{(n_q + 1)(1 - m + \alpha)}{m - \alpha}}\,\Upsilon_D^{\lambda_q}(n_q).
\end{align*}
The exponential decay is now governed by 
\begin{equation}\label{eq:nmb1}
\Theta_{n_q,m,j}^{\alpha}(x,\eta) = \left(\frac{j^2 x \eta^{\frac{1 - m + \alpha}{m - \alpha}}}{4 n_q^2 (m - \alpha)}\right)^{n_q},
\end{equation}
requiring
\begin{subequations}
\begin{gather}
j < n_q,\quad\text{or}\label{eq:convcondver2}\\
j = n_q\quad\text{and}\quad\frac{x \eta^{\frac{1 - m + \alpha}{m - \alpha}}}{4 (m - \alpha)} < 1,\label{eq:convcondver22}
\end{gather}
\end{subequations}
which contradict the assumption $j \gg n_q$; hence, the error diverges in this case.
\end{description}

Under the assumption $m \ll n_q$, the error analysis of Theorem \ref{subsec:err:thm2} shows that the quadrature truncation error converges only when $n_q \le n-m-k: 1 \le k \ll n_q$. In the special case when $n_q > n-m-1$, the quadrature truncation error totally collapse by Theorem~\ref{subsec:err:thm3}. In all cases, the parameter \(\lambda\) always serve as a decay accelerator, whereas \(\lambda_q\) functions as a decay brake. Notably, the observed slower convergence rate with increasing \(\lambda_q\) aligns well with the earlier finding in \cite{Elgindy201382} that selecting relatively large positive values of \(\lambda_q > 2\) causes the Gegenbauer weight function associated with the GIM to diminish rapidly near the boundaries \(x = \pm 1\). This effect shifts the focus of the Gegenbauer quadrature toward the central region of the interval, increasing sensitivity to errors and making the quadrature more extrapolatory.  

The following theorem provides an upper bound for the asymptotic total error, encompassing both the series truncation error and the quadrature approximation error in light of Theorems \ref{thm:csbhvs12} and \ref{subsec:err:thm2}. 

\begin{thm}[Asymptotic Total Truncation Error Bound]\label{thm:TotErrKKK1}
Let $m \ll n_q$, and suppose that the assumptions of Theorem \ref{subsec:err:thm1} hold true. Then the total truncation error, denoted by ${}^{\alpha}\ME{E}_{n, n_q}^{\lambda, \lambda_q}(x, \xi)$, arising from both the series truncation \eqref{sec:ort:eq:Lagint1} and the quadrature approximation \eqref{eq:Solynn1nn1}, is asymptotically bounded above by:
\begin{gather}
\left|{}^{\alpha}\ME{E}_{n, n_q}^{\lambda, \lambda_q}(x, \xi, \eta)\right| 
\simlt \,A {\hvartheta _{m ,\lambda }}{2^{ - 2\lambda  - 2n - \frac{3}{2}}}{n^{\lambda  + m}} \\
+ \frac{\varpi^{\text{upp}} B_m^{\lambda ,{\lambda _q}} x^{m-\alpha} \left\|f\right\|_{L^{\infty}(\FOmega_1)}}{\lambdabar_{\max}^{\lambda}\,\Gamma(m-\alpha+1)} {2^{- 2{n_q}}} {n^{2(1 - m - \lambda)}} n_q^{- 2{n_q} - m - \lambda  + {\lambda _q} - \frac{5}{2}} \times\\
{\left( {n - {n_q}} \right)^{- n + m + {n_q}+\frac{3}{2}}} {\left( {n + {n_q}} \right)^{n + 2\lambda + m + {n_q}+\frac{1}{2}}} \Lambda_{n_q,m}^{\alpha}(x,\eta) \times\\
{}_2\Upsilon_{\sigma,D}^{\lambda,\lambda_q}(n,n_q)\,\MFI_{n \ge m+n_q+1},
\end{gather}
$\forallL n, n_q$, where  
\begin{equation}
\varpi^{\text{upp}} = \left\{ \begin{array}{l}
\varpi^{\text{upp},+},\quad \forall \lambda \in \MBRzerP,\\
\varpi^{\text{upp,-}},\quad \lambda \in \MBRmhzer,\\
\end{array} \right.
\end{equation}
\begin{equation}
\frac{1}{\lambdabar_{\max}^{\lambda}} = \left\{ \begin{array}{l}
\displaystyle{\frac{1}{\lambdabar_n^{\lambda}}},\quad \lambda \in \MBRzerP,\\
\displaystyle{\frac{1}{\lambdabar_{m+n_q+1}^{\lambda}}},\quad \lambda \in \MBRmhzer,
\end{array} \right.
\end{equation}
\begin{equation}
{}_2\Upsilon_{\sigma,D}^{\lambda,\lambda_q}(n,n_q) = \begin{cases}
1, & \lambda \in \MBRzerP,\quad \lambda_q \in \MBRzerP, \\
\sigma^{\lambda} n^{-\lambda}, & \lambda \in \MBRmhzer,\quad \lambda_q \in \MBRzerP, \\
D^{\lambda_q} n_q^{-\lambda_q}, & \lambda \in \MBRzerP,\quad \lambda_q \in \MBRmhzer, \\
\sigma^{\lambda} D^{\lambda_q} n^{-\lambda} n_q^{-\lambda_q}, & \lambda \in \MBRmhzer,\quad \lambda_q \in \MBRmhzer,
\end{cases}
\end{equation}
$A, {\hvartheta _{m ,\lambda }}, D^{\lambda_q}$, and $\sigma^{\lambda}$ are constants with the definitions and properties outlined in Theorems \ref{thm:csbhvs12} and \ref{subsec:err:thm2}, as well as in Eq. \eqref{eq:sgvcmfhm1}.
\end{thm}
\begin{proof}
The total truncation error is the sum of the truncation error associated with Caputo FD approximation \eqref{eq:Soly1_2}, ${}^{\alpha}\ME{T}_{n}^{\lambda}(x, \xi)$, and the accumulated truncation errors associated with the quadrature approximation \eqref{eq:Solynn1nn1}, for $j = m:n$, arising from Formula \eqref{eq:Soly1_2}:
\begin{gather}
{}^{\alpha}\ME{E}_{n, n_q}^{\lambda, \lambda_q}(x, \xi, \eta) 
= {}^{\alpha}\ME{T}_{n}^{\lambda}(x, \xi)\\
 + \frac{x^{m-\alpha}}{\Gamma(m-\alpha+1)} \sum_{k \in \MBJP_n} {\hvarpi_k^{\lambda} f_k \sum_{j \in \MBN_{m:n}} {\left(\hlambdabar_j^{\lambda}\right)^{-1} \FR{T}_{j, n_q}^{\lambda_q}(\eta)\,\hG_j^{\lambda} \left(\hx_{n,k}^{\lambda}\right)}}\\
  = {}^{\alpha}\ME{T}_{n}^{\lambda}(x, \xi) + \frac{x^{m-\alpha}}{\Gamma(m-\alpha+1)} \sum_{k \in \MBJP_n} {\varpi_k^{\lambda} f_k \sum_{j \in \MBN_{m:n}} {\left(\lambdabar_j^{\lambda}\right)^{-1} \FR{T}_{j, n_q}^{\lambda_q}(\eta)\,\hG_j^{\lambda} \left(\hx_{n,k}^{\lambda}\right)}},
\end{gather}
where $\FR{T}_{j, n_q}^{\lambda_q}(\eta)$ is the truncation error associated with the quadrature approximation \eqref{eq:Solynn1nn1} $\,\foralle j$, and ${\lambdabar}_{0:n}^{\lambda}$ and $\varpi_{0:n}^{\lambda}$ are the normalization factors for Gegenbauer polynomials and the Christoffel numbers associated with their quadratures. The key upper bounds on these latter factors were recently derived in \cite[Lemmas B.1 and B.2]{Elgindy2023d}:
\begin{gather}
\varpi _{j}^\lambda \simlteq \varpi^{\text{upp},+} = \frac{\pi }{{n + 1}}\quad \forall (j,\lambda) \in \,\MBJP_n \times \MBRzerP,\\
\varpi _{j}^\lambda < \varpi^{\text{upp,-}} = \frac{\Gamma^2(\lambda+1/2)}{2\, n^{1+2\lambda}}\quad \forall (j,\lambda) \in \,\MBJP_n \times \MBRmhzer,\\
\max\limits_{j \in \MBJP_n} {\frac{1}{\lambdabar_j^{\lambda}}} = \left\{ \begin{array}{l}
\displaystyle{\frac{1}{\lambdabar_n^{\lambda}}},\quad \lambda \in \MBRzerP,\\
\displaystyle{\frac{1}{\lambdabar_0^{\lambda}}},\quad \lambda \in \MBRmhzer,
\end{array} \right.
\end{gather}
where $\lambdabar_0^{\lambda} = \displaystyle{\frac{\sqrt{\pi}\,\Gamma(1/2+\lambda)}{\Gamma(1+\lambda)}}$. By combining these results with Eq. \eqref{eq:sgvcmfhm1}, we can bound the total truncation error by
\begin{gather}
\left|{}^{\alpha}\ME{E}_{n, n_q}^{\lambda, \lambda_q}(x, \xi, \eta)\right| 
\simlt \,A {\hvartheta _{m ,\lambda }}{2^{ - 2\lambda  - 2n - \frac{3}{2}}}{n^{\lambda  + m}} \\
+ \frac{\varpi^{\text{upp}} x^{m-\alpha} \left\|f\right\|_{L^{\infty}(\FOmega_1)}}{\lambdabar_{\max}^{\lambda}\,\Gamma(m-\alpha+1)} (n+1) (n-m-n_q) \max\limits_{j \in \MBN_{m+n_q+1:n}} \left|\FR{T}_{j, n_q}^{\lambda_q}(\eta)\right|\,\Upsilon_{\sigma}^{\lambda}(n),\label{eq:trryr1}
\end{gather}
$\forallL n$. Since the $j$-dependent polynomial factor
\[(j - m - n_q)^{-j + m + n_q + \frac{1}{2}} j^{-2\lambda - 2m + 1} (j + n_q)^{j + 2\lambda + m + n_q + \frac{1}{2}},\]
is maximized at $j = n$ by Lemma \ref{lem:cbsdc1}, the proof is accomplished by applying the asymptotic inequality \eqref{eq:SolySofy20251} on \eqref{eq:trryr1} after replacing $j$ with $n$.
\end{proof}

Under the assumptions of Theorem~\ref{thm:TotErrKKK1}, exponential error decay dominates the overall error behavior if $n_q \le n-m-k: 1 \le m, k \ll n_q$, provided that Convergence Condition~\eqref{eq:convcondver1} holds. In the special case when $n_q > n - m - 1$, the total truncation error reduces to pure interpolation error, as the quadrature truncation error vanishes. The rigorous asymptotic analysis presented in this section leads to the following practical guideline for selecting $\lambda$ and $\lambda_q$:\\

\textbf{Rule of Thumb} \textit{(Selection of $\lambda$ and $\lambda_q$ Parameters)}.
\begin{itemize}
\item \textit{High-precision computations}: Consider $\lambda \in (0, 2]$ with appropriately adjusted $\lambda_q$:
\[
-1/2 + \varepsilon \le \lambda_q \le \lambda + 2m + 1\quad \forallS m.
\]
\item \textit{General-purpose computations}: Consider $\lambda = \lambda_q = 0$ (SC quadrature) $\forallL n \text{ and } n_q$. This latter choice is motivated by the fact that the truncated expansion in the SC quadrature is known to be optimal in the $L^{\infty}$-norm for definite integral approximations of smooth functions. 
\end{itemize}

\begin{rem}
It is important to note that the observations made in this section rely on asymptotic results $\forallL n, n_q$. However, since the integrand is smooth when $\alpha \not\approx m$, the SG quadrature often achieves high accuracy with relatively few nodes. Smooth integrands may exhibit spectral convergence before asymptotic effects takes place as we demonstrate later in Section \ref{sec:CRAC1}.
\end{rem}

\section{Case Study: Caputo Fractional TPBVP of the Bagley–Torvik Type}
\label{sec:PS}
In this section, we consider the application of the proposed method on the following Caputo fractional TPBVP of the\\ Bagley–Torvik type defined as follows:
\begin{subequations}
\begin{equation}\label{eq:AD1}
a \CapD{x}{\alpha}{u} + b \CapD{x}{1.5}{u} + c u(x) = f(x),\quad x \in \FOmega_1,
\end{equation}
with the given Dirichlet boundary conditions
\begin{equation}\label{eq:AD2}
u(0) =  \gamma_1,\quad u(1) = \gamma_2,
\end{equation}
\end{subequations}
where $\alpha > 1, \{a,b,c,\gamma_{1:2}\} \subset \MBR$, and $f \in L^2(\FOmega_1)$. With the derived numerical instrument for approximating Caputo FDs, determining an accurate numerical solution of the TPBVP is rather straightforward. Indeed, collocating System \eqref{eq:AD1} at the SGG set $\left\{\hx_{n,0:n}^{\lambda}\right\} = \MBSG_n^{\lambda}$ in conjunction with Eq. \eqref{eq:mfkfj2} yield
\begin{subequations}
\begin{equation}\label{eq:zxncnscv1}
a \CapIM{n}{\alpha}\,u_{0:n} + b \CapIM{n}{1.5}\,u_{0:n} + c u_{0:n} = f_{0:n}.
\end{equation} 
Since $\hG_k^{\lambda}(0) = (-1)^k$ and $\hG_k^{\lambda}(1) = 1\,\forall k \in \MBJP_n$, by the properties of SG polynomials, substituting the boundary conditions \eqref{eq:AD2} into Eq. \eqref{sec:ort:eq:Lagint1} give the following system of equations:
\begin{gather}
	\left[\trp{{\hlambdabar_{0:n}^{\lambda}}^{\hspace{-1mm} \div}}\,\left(\left((-1)^{0:n} \otimes \bmone_{n+1}\right)^{\top} \odot \hG_{0:n}^{\lambda}\cancbra{\hbmx_n^{\lambda}}\right) \Diag{\hvarpi_{0:n}^{\lambda}}\right] u_{0:n} = \gamma_1,\label{eq:dfdffdvf1}\\
		\left[\trp{{\hlambdabar_{0:n}^{\lambda}}^{\hspace{-1mm} \div}}\,\hG_{0:n}^{\lambda}\cancbra{\hbmx_n^{\lambda}} \Diag{\hvarpi_{0:n}^{\lambda}}\right] u_{0:n} = \gamma_2.\label{eq:dfdffdvf2}
\end{gather}
\end{subequations}
Therefore, the linear system described by Eqs. \eqref{eq:zxncnscv1}, \eqref{eq:dfdffdvf1}, and \eqref{eq:dfdffdvf2} can now be compactly written in the following form:
\begin{equation}\label{eq:Finally1}
\F{A} u_{0:n} = \bm{F},
\end{equation}
where 
\begin{equation}
\F{A} = \begin{bmatrix}
a \CapIM{n}{\alpha} + b \CapIM{n}{1.5} + c \F{I}_{n+1}\\
\trp{{\hlambdabar_{0:n}^{\lambda}}^{\hspace{-1mm} \div}}\,\left(\left((-1)^{0:n} \otimes \bmone_{n+1}\right)^{\top} \odot \hG_{0:n}^{\lambda}\cancbra{\hbmx_n^{\lambda}}\right) \Diag{\hvarpi_{0:n}^{\lambda}}\\
\trp{{\hlambdabar_{0:n}^{\lambda}}^{\hspace{-1mm} \div}}\,\hG_{0:n}^{\lambda}\cancbra{\hbmx_n^{\lambda}} \Diag{\hvarpi_{0:n}^{\lambda}}
\end{bmatrix},
\end{equation}
and 
\begin{equation}
\bm{F} = [f_{0:n}^{\top}, \gamma_{1:2}]^{\top}.
\end{equation}
The solution of the linear system \eqref{eq:Finally1} provides the approximate solution values at the SGG points. The solution values at any non-collocated point in $\FOmega_1$ can further be estimated with excellent accuracy via the interpolation formula \eqref{sec:ort:eq:Lagint1}. 

When $\alpha \in \MBZP$, Caputo FD reduces to the classical integer-order derivative of the same order. In this case, we can use the first-order GDM in barycentric form, $\F{D}^{(1)}$, of \citet{elgindy2018highb}. This matrix enables the approximation of the function's derivative at the GG nodes using the function values at those nodes by employing matrix-vector multiplication. The entries of the differentiation matrix are computed based on the barycentric weights and GG nodes. The associated differentiation formula exhibits high accuracy, often exhibiting exponential convergence for smooth functions. This rapid convergence is a hallmark of PS methods and makes the GDM highly accurate for approximating derivatives. Furthermore, the utilization of barycentric forms improves the numerical stability of the differentiation matrix and lead to efficient computations. Using the properties of PS differentiation matrices, higher-order differentiation matrices can be readily generated through successive multiplication by the first-order GDM:
\begin{equation}
\F{D}^{(k)} = \F{D}^{(1)}_{(k)},\quad \forall k > 1.
\end{equation}
The SGDM of any order $k, \hFD^{(k)}$, based on the SGG points set $\MBSG_n^{\lambda}$, can be generated directly from $\F{D}^{(1)}$ using the following formula:
\begin{equation}
\hFD^{(k)} = 2^k\,\F{D}^{(1)}_{(k)},\quad \forall k \ge 1.
\end{equation}

\section{Numerical Examples}
\label{sec:CRAC1}
In this section, we present the numerical experiments conducted on a personal laptop equipped with an AMD Ryzen 7 4800H processor (2.9 GHz, 8 cores/16 threads), 16GB of RAM, and running Windows 11. All simulations were performed using MATLAB R2023b. The accuracy of the computed solutions was assessed using absolute errors and maximum absolute errors, which provide quantitative measures of the pointwise and worst-case discrepancies between the exact and numerical solutions, respectively.

\textbf{Example 1.} Consider the following Caputo fractional TPBVP of the Bagley–Torvik type
\begin{subequations}
\begin{equation}\label{eq:AD1sdgx}
\CapD{x}{2}{u} + \CapD{x}{1.5}{u} + u(x) = x^2 + 2 + 4 \sqrt{\frac{x}{\pi}},\quad x \in \FOmega_1,
\end{equation}
with the given Dirichlet boundary conditions
\begin{equation}\label{eq:AD2sdbdf}
u(0) = 0,\quad u(1) = 1.
\end{equation}
\end{subequations}
The exact solution is $u(x) = x^2$. This problem was solved by \citet{al2010collocation} using a method that combines conjugating collocation, spline analysis, and the shooting technique. Their reported error norm was $3.78 \times 10^{-12}$; cf. \cite{batool2025fractional}. Later, \citet{batool2025fractional} addressed the same problem using integral operational matrices based on Chelyshkov polynomials, transforming the problem into solvable Sylvester-type equations. They reported an error norm of $2.3388 \times 10^{-25}$, obtained using approximate solution terms with significantly more than 16 digits of precision. Specifically, the three terms used to derive this error included 32, 47, and 47 digits after the decimal point, indicating that the method utilizes extended or arbitrary-precision arithmetic, rather than being constrained to standard double precision. For a more fair comparison, since all components of our computational algorithm adhere to double-precision representations and computations, we recalculated their approximate solution using \cite[Eq.~(92)]{batool2025fractional} on the MATLAB platform with double-precision arithmetic. Our results indicate that the maximum absolute error in their approximate solution, evaluated at 50 equally spaced points in $\FOmega_1$, was approximately $2.22 \times 10^{-16}$. The SGPS method produced this same result using the parameters $n = n_q = 4$ and $\lambda = \lambda_q = 1.1$. The elapsed time required to run the SGPS method was $0.004732$ seconds. Figure~\ref{fig:figJan1} illustrates the exact solution, the approximate solution obtained using the SGPS method, and the absolute errors at the SGG collocation points.  

\begin{figure}[ht]
 \centering
 \includegraphics[scale=0.45]{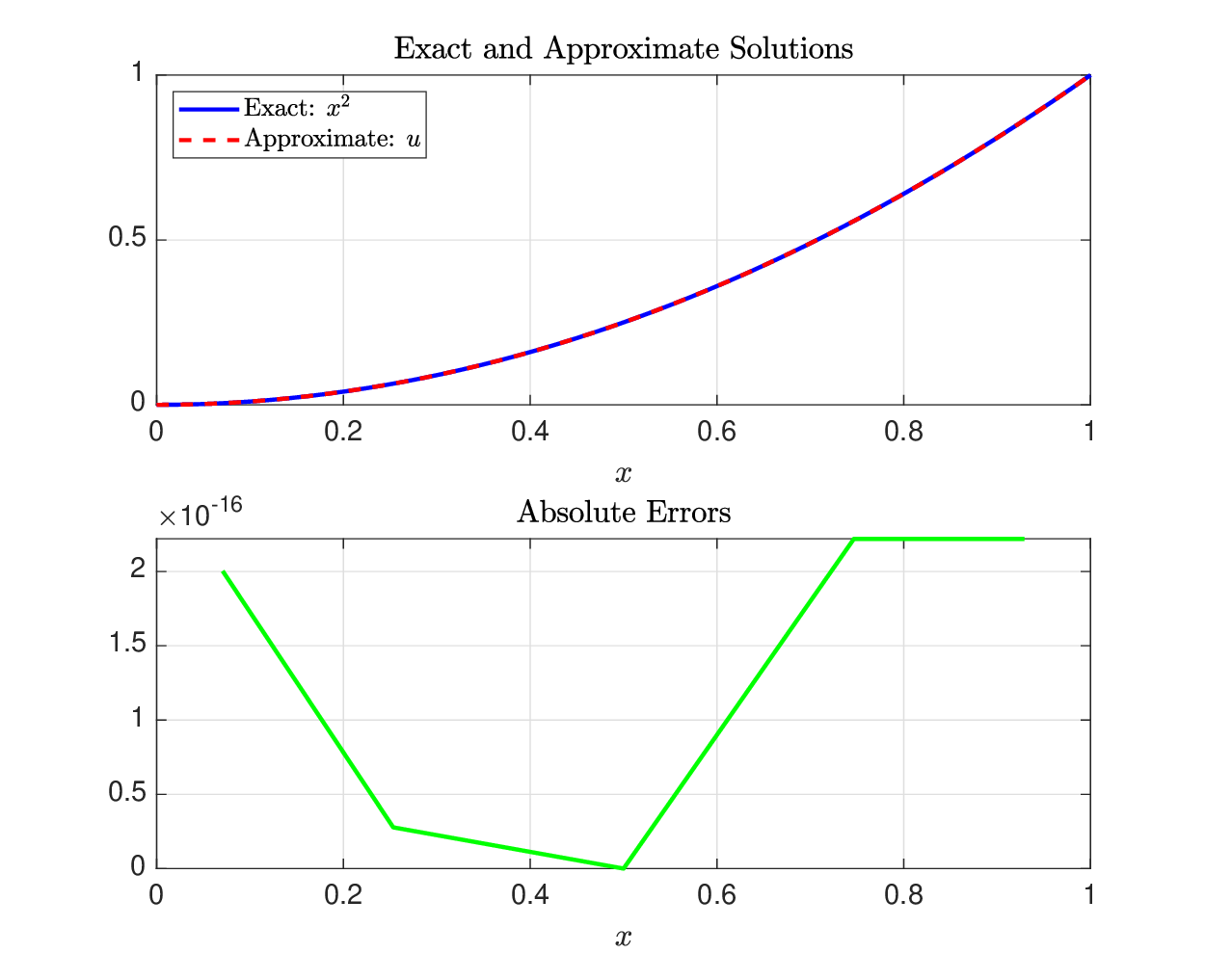}
 \caption{The exact solution of Example 1 and its approximation on $\FOmega_1$ (upper) and the absolute errors at the collocation points (lower). The approximate solution was obtained using the SGPS method with parameters $n = n_q = 4$ and $\lambda = \lambda_q = 1.1$.}
 \label{fig:figJan1}
\end{figure}

\textbf{Example 2.} Consider the following Caputo fractional TPBVP of the Bagley–Torvik type
\begin{subequations}
\begin{equation}\label{eq:AD1zazaza}
\CapD{x}{2}{u} + \CapD{x}{1.5}{u} + u(x) = 1+x,\quad x \in \FOmega_1,
\end{equation}
with the given Dirichlet boundary conditions
\begin{equation}\label{eq:AD2xsxsxsxs}
u(0) = 1,\quad u(1) = 2.
\end{equation}
\end{subequations}
The exact solution is $u(x) = 1+x$. \citet{yuzbacsi2013numerical} solved this problem using a numerical technique based on collocation points, matrix operations, and a generalized form of Bessel functions of the first kind. The maximum absolute error reported in \cite{yuzbacsi2013numerical} (at $M = 6$) was $4.6047 \times 10^{-8}$. Our SGPS method produced near exact solution values within a maximum absolute error of $4.44 \times 10^{-16}$ using $n = n_q = \lambda = \lambda_q = 2$; cf. Figure \ref{fig:figJan2}. The elapsed time required to run the SGPS method was $0.004142$ seconds. 

\begin{figure}[ht]
 \centering
 \includegraphics[scale=0.45]{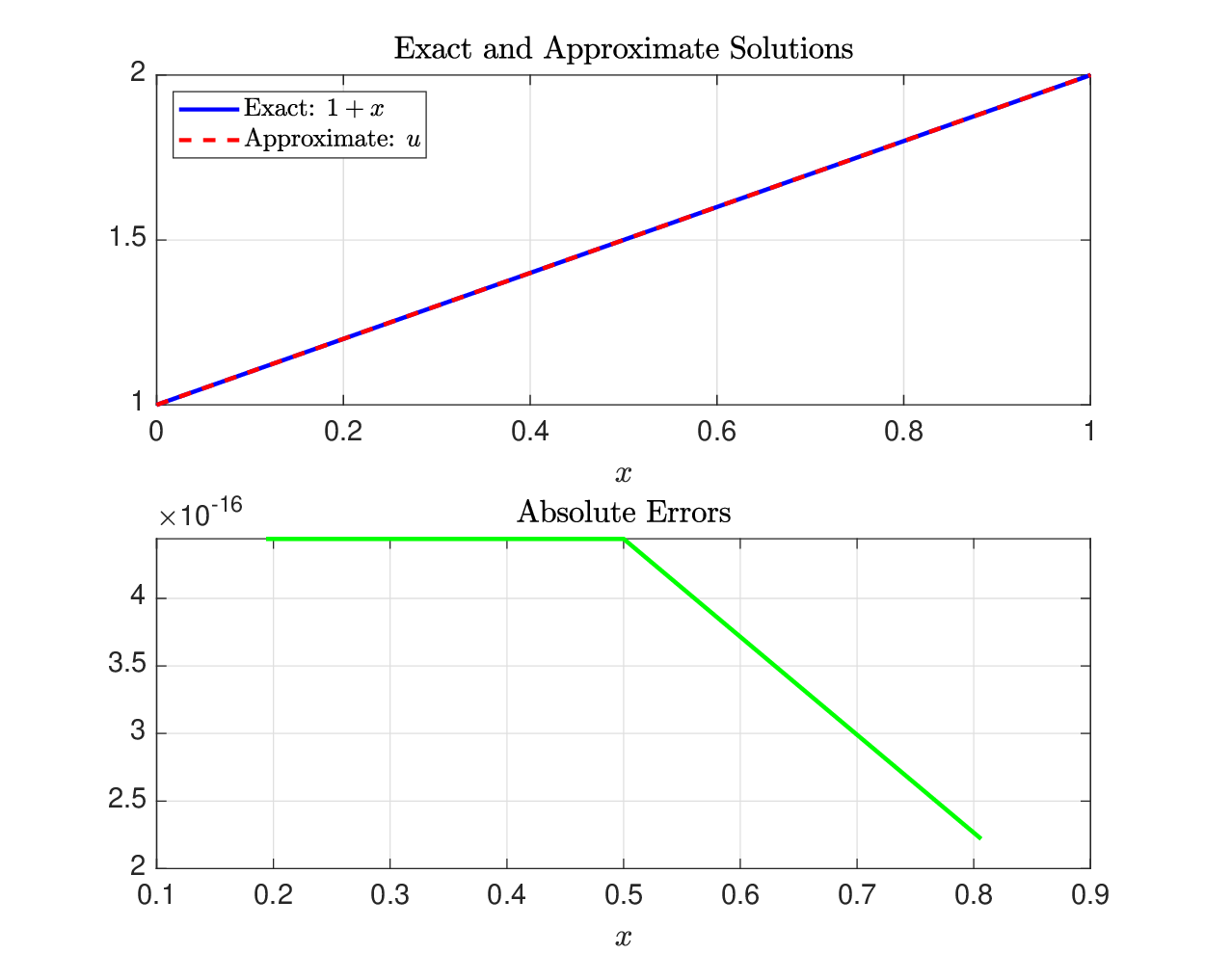}
 \caption{The exact solution of Example 2 and its approximation on $\FOmega_1$ (upper) and the absolute errors at the collocation points (lower). The approximate solution was obtained using the SGPS method with parameters $n = n_q = \lambda = \lambda_q = 2$.}
 \label{fig:figJan2}
\end{figure}

\textbf{Example 3.} Consider the following Caputo fractional TPBVP of the Bagley–Torvik type
\begin{subequations}
\begin{equation}\label{eq:AD1qwqwqwq}
\CapD{x}{1.5}{u} + u(x) = \frac{2}{\Gamma(3/2)} \sqrt{x} + x (x-1),\quad x \in \FOmega_1,
\end{equation}
with the given Dirichlet boundary conditions
\begin{equation}\label{eq:AD2wewewe}
u(0) = u(1) = 0.
\end{equation}
\end{subequations}
The exact solution is $u(x) = x^2-x$. Our SGPS method produced near exact solution values within a maximum absolute error of $1.94 \times 10^{-16}$ using $n = n_q = 3$ and $\lambda = \lambda_q = 1$; cf. Figure \ref{fig:figJan3}. The elapsed time required to run the SGPS method was $0.004160$ seconds. 

\begin{figure}[ht]
 \centering
 \includegraphics[scale=0.45]{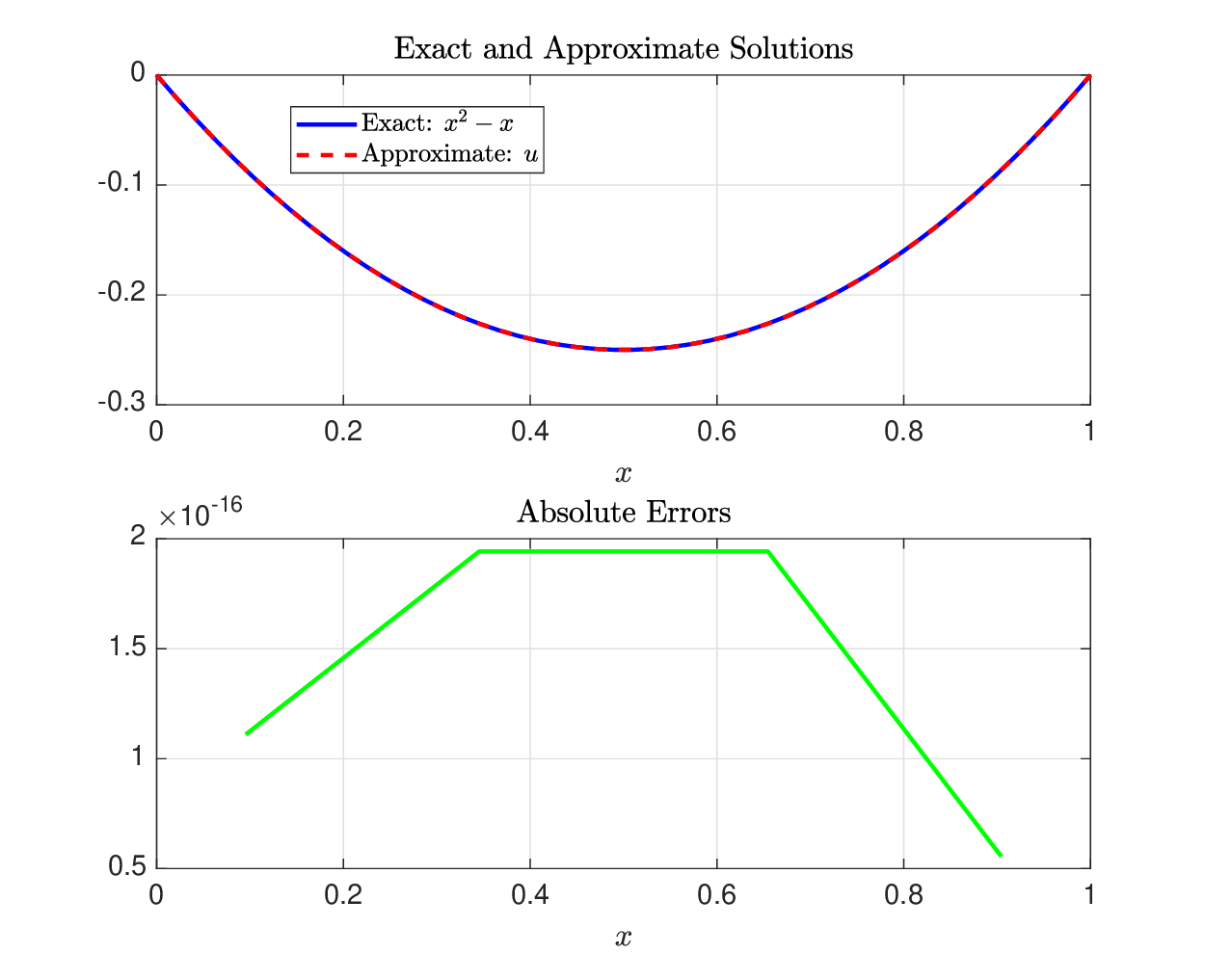}
 \caption{The exact solution of Example 3 and its approximation on $\FOmega_1$ (upper) and the absolute errors at the collocation points (lower). The approximate solution was obtained using the SGPS method with parameters $n = n_q = 3$ and $\lambda = \lambda_q = 1$.}
 \label{fig:figJan3}
\end{figure}

\section{Conclusion and Discussion}
\label{sec:Conc}
This study introduced an SGPS method for approximating the Caputo FD of higher orders and solving the TPBVPs of the Bagley–Torvik type. The proposed SGPS method introduces several novel contributions to the field of fractional numerical analysis, such as (i) the development of FSGIMs that accurately and efficiently approximate Caputo FDs at any random set of points using SG quadratures. This approach generalizes traditional pseudospectral differentiation matrices to the fractional-order setting, which we consider a significant theoretical advancement. (ii) The use of FSGIMs allow for pre-computation and storage, significantly accelerating execution of the SGPS method. (iii) The method applies an innovative change of variables that transforms the Caputo FD into a scaled integral of an integer-order derivative. This transformation simplifies computations, facilitates error analysis, and mitigates singularities in the Caputo FD near zero, which improves both stability and accuracy, (iv) the method can produce approximations withing near full machine precision at an exponential rate using relatively coarse mesh grids, (v) the method generally improves numerical stability and attempts to avoid issues related to ill-conditioning in classical pseudospectral differentiation matrices by using SG quadratures in barycentric form, (vi) the proposed methodology can be extended to multidimensional fractional problems, making it a strong candidate for future research in high-dimensional fractional differential equations, (vii) unlike traditional methods that treat interpolation and integration separately, the current method unifies these operations into a cohesive framework using SG polynomials. Numerical experiments validated the superior accuracy of the proposed method over existing techniques, achieving near-machine precision results in many cases. The current study also highlighted critical guidelines for selecting the parameters $\lambda$ and $\lambda_q$ to optimize the performance of the SGPS method for large interpolation and quadrature mesh sizes and relatively small fractional order $\alpha$. For high-precision computations, $\lambda$ should be chosen within the range $(0, 2]$, with $\lambda_q$ adjusted such that $-1/2 + \varepsilon \leq \lambda_q \leq \lambda + 2m + 1$, where $m = \left\lceil  \alpha  \right\rceil$. This ensures a balance between convergence speed and numerical stability. For general-purpose computations, setting $\lambda = \lambda_q = 0$ (corresponding to the SC quadrature) is recommended, as it provides optimal $L^\infty$-norm accuracy for smooth functions. The analysis also revealed that increasing $\lambda$ accelerates theoretical convergence but may introduce numerical instability due to extrapolation effects, while larger $\lambda_q$ values can slow convergence. These insights ensure robust and efficient implementations of the SGPS method across diverse problem settings. Future work may focus on extending the SGPS approach to multidimensional fractional problems alongside adaptive quadrature for unconditional convergence in finite-precision settings.

\section*{Declarations}
\subsection*{Competing Interests}
The author declares there is no conflict of interests.

\subsection*{Availability of Supporting Data}
The author declares that the data supporting the findings of this study are available within the article.

\subsection*{Ethical Approval and Consent to Participate and Publish}
Not Applicable.

\subsection*{Human and Animal Ethics}
Not Applicable.

\subsection*{Consent for Publication}
Not Applicable.

\subsection*{Funding}
The author received no financial support for the research, authorship, and/or publication of this article.

\subsection*{Authors' Contributions}
The author confirms sole responsibility for the following: study conception and design, data collection, analysis and interpretation of results, and manuscript preparation.


\appendix
\section{List of Acronyms}
\label{sec:LVA1}
\begin{table}[H]
    \centering
    \caption{List of Acronyms}
\resizebox{0.4\textwidth}{!}{%
    \begin{tabular}{|c|c|}
        \hline
        \textbf{Acronym} & \textbf{Meaning} \\
        \hline
        FD & Fractional derivative \\
        FSGIM & Fractional-order shifted Gegenbauer integration matrix \\
		GDM & Gegenbauer differentiation matrix \\
        GIM & Gegenbauer integration matrix \\
        GIRV & Gegenbauer integration row vector \\
        SC & Shifted Chebyshev \\
	    SGDM & Shifted Gegenbauer differentiation matrix \\
        SGIM & Shifted Gegenbauer integration matrix \\
        SGIRV & Shifted Gegenbauer integration row vector \\
        SGPS & Shifted Gegenbauer pseudospectral \\
        SG & Shifted Gegenbauer \\
        SGG & Shifted Gegenbauer-Gauss \\
        TPBVP & Two-point boundary value problem \\
        \hline
    \end{tabular}
    }
    \label{tab:acronyms}
\end{table}

\section{Mathematical Proof}
\label{sec:MP1}
\begin{lem}\label{lem:hello1}
Let \(\lambda > -\frac{1}{2} , m \ge 1 \), and \( j \geq m + n_q + 1 \). Then, the 
\( j \)-dependent factor in Eq. \eqref{eq:bnmbn1}:
\begin{equation}\label{eq:keyratio1}
   \frac{(j - m)!\,\Gamma(j + m + n_q + 2\lambda + 1)}{\Gamma(j + m + 2\lambda)\,\Gamma(j - m - n_q)},
\end{equation}
has the asymptotic order $O\left(j^{2n_q + 2}\right)$ as $j \to \infty, \forallL n_q$.
\end{lem}
\begin{proof}
We analyze the asymptotic behavior of the expression as \( j \to \infty \) using Stirling's approximation for the Gamma function:
\[
\Gamma(z) \approx \sqrt{2\pi} \, z^{z - \frac{1}{2}} e^{-z}\quad \forallL z.
\]
By realizing also that \( (j - m)! = \Gamma(j - m + 1) \), we have:
   \[
   \begin{aligned}
   \Gamma(j - m + 1) &\approx \sqrt{2\pi} \, j^{j - m + \frac{1}{2}} e^{-j}, \\
   \Gamma(j + m + n_q + 2\lambda + 1) &\approx \sqrt{2\pi} \, (j + n_q)^{j + m + n_q + 2\lambda + \frac{1}{2}} e^{-j - n_q}, \\
   \Gamma(j + m + 2\lambda) &\approx \sqrt{2\pi} \, j^{j + m + 2\lambda - \frac{1}{2}} e^{-j}, \\
   \Gamma(j - m - n_q) &\approx \sqrt{2\pi}\,(j - n_q)^{j - m - n_q - \frac{1}{2}} e^{n_q-j}.
   \end{aligned}
   \]
Since \((j \pm n_q)^k \approx j^k (1 \pm \frac{n_q}{j})^k \approx j^k e^{\pm k n_q/j}\,\forallL j\), we can write the key ratio \eqref{eq:keyratio1} as follows:
\begin{gather}
\frac{\Gamma(j - m + 1)\Gamma(j + m + n_q + 2\lambda + 1)}{\Gamma(j + m + 2\lambda)\Gamma(j - m - n_q)}\\
\approx \frac{j^{j - m + \frac{1}{2}}(j + n_q)^{j + m + n_q + 2\lambda + \frac{1}{2}}}{j^{j + m + 2\lambda - \frac{1}{2}}(j - n_q)^{j - m - n_q - \frac{1}{2}}} e^{-2n_q}\\
\approx \frac{j^{j - m + \frac{1}{2}} j^{j + m + n_q + 2\lambda + \frac{1}{2}} e^{(j + m + n_q + 2\lambda + \frac{1}{2})n_q/j}}{j^{j + m + 2\lambda - \frac{1}{2}} j^{j - m - n_q - \frac{1}{2}} e^{-(j - m - n_q - \frac{1}{2})n_q/j}} e^{-2n_q}\\
= {e^{\frac{{2\lambda {n_q}}}{j}}}{j^{2 + 2{n_q}}} = O\left(j^{2n_q + 2}\right),\quad \text{as }j \to \infty, \forallL n_q.
\end{gather}
\end{proof}

The following lemma is useful in analyzing the error bound of Theorem \ref{subsec:err:thm2}.

\begin{lem}\label{lem:cbsdc1}
Let $\lambda > -\frac{1}{2}$ and $m \geq 1$ be an integer. The function
\[
E(j) = j^{-2\lambda - 2m + 1} (j - m - n_q)^{-j + m + n_q + \frac{1}{2}} (j + n_q)^{j + 2\lambda + m + n_q + \frac{1}{2}},
\]
is strictly increasing with $j\,\forall j \ge m+n_q+1\,\forallL n_q$.
\end{lem}
\begin{proof}
Suppose that the assumptions of the lemma hold true. We show first that the logarithmic derivative of $E(j)$ is positive $\,\forall j \ge m+n_q+1$. To this end, take the natural logarithm:
\[
\ln E(j) = A\ln j + B\ln(j - m - n_q) + C\ln(j + n_q),
\]
where
\begin{align*}
A &= -2\lambda - 2m + 1, \\
B &= -j + m + n_q + \tfrac{1}{2}, \\
C &= j + 2\lambda + m + n_q + \tfrac{1}{2}.
\end{align*}
Differentiating with respect to $j$ yields:
\begin{gather}
\Part{j}\ln E(j) = \frac{A}{j} - \ln(j - m - n_q) + \frac{B}{j - m - n_q} + \ln(j + n_q) + \frac{C}{j + n_q} \\
= \ln\left(\frac{j + n_q}{j - m - n_q}\right) + \frac{A}{j} + \frac{B}{j - m - n_q} + \frac{C}{j + n_q}.
\end{gather}
For $j \geq m + n_q + 1$, we have:
\begin{itemize}
\item $\displaystyle{\ln\left(\frac{j + n_q}{j - m - n_q}\right)} > 0$, since $\displaystyle{\frac{j + n_q}{j - m - n_q}} > 1$.
\item $\displaystyle{\frac{A}{j} \to 0^{-}}$.
\item $\displaystyle{\frac{B}{j - m - n_q} = -1 + \frac{1}{2 (j-m-n_q)} \in (-1,-1/2]}$.
\item $\displaystyle{\frac{C}{j + n_q} = 1 + \frac{2\lambda+m+1/2}{j+n_q} = 1^+.}$
\end{itemize}
The rational terms combine to give a positive quantity. Thus, the logarithmic derivative, $\Part{j}\ln E(j)$, is positive $\,\forall j \ge m+n_q+1$. Since the natural logarithm is strictly increasing, it follows that $E(j)$ itself must be strictly increasing with $j$ in that range.
\end{proof}

\bibliographystyle{model1-num-names}
\bibliography{Bib}
\end{document}